\documentclass[11pt]{article}
 \expandafter\let\csname equation*\endcsname\relax
 \expandafter\let\csname endequation*\endcsname\relax
\usepackage{bm}
\usepackage{enumerate}
\usepackage[top=1in, bottom=1in, left=1in, right=1in]{geometry}
\usepackage{mathrsfs,color}
\usepackage{amsfonts}
\usepackage{amssymb}
\usepackage{dsfont}
\usepackage{todonotes}
\usepackage{caption,comment}
\usepackage{blkarray}
\usepackage{amsmath}
\usepackage{float}
\usepackage{amssymb,amsthm}
\usepackage[toc,page]{appendix}
\usepackage{graphicx,afterpage}
\usepackage{epstopdf}
\usepackage{cancel}
\usepackage{mathdots}
\usepackage[pdftex,colorlinks=true]{hyperref}
\usepackage[sort,compress]{cite}

\usepackage{todonotes}

\usepackage{mhequ}
\hypersetup{CJKbookmarks,%
bookmarksnumbered,%
colorlinks,%
linkcolor=black,%
citecolor=black,%
plainpages=false,%
pdfstartview=FitH}
\numberwithin{equation}{section}
\numberwithin{figure}{section}

\newcommand\tabcaption{\def\@captype{table}\caption}

\newtheorem{thm}{Theorem}[section]
\newtheorem{cor}[thm]{Corollary}
\newtheorem{lem}[thm]{Lemma}

\newcommand{\bfx}{\mathbf{x}}

\newcommand{\bfI}{\mathbf{I}}
\newcommand{\bfz}{\mathbf{z}}

\newcommand{\bfX}{\mathbf{X}}
\newcommand{\bfw}{\mathbf{w}}

\newcommand{\E}{\mathbb{E}}
\newcommand{\Prob}{\mathbb{P}}

\newcommand{\reals}{\mathbb{R}}

\newcommand{\unit}{\mathds{1}}

\newcommand{\bfd}{\mathbf{d}}

\newcommand{\bfY}{\mathbf{Y}}
\newcommand{\bff}{\mathbf{f}}
\newcommand{\bfh}{\mathbf{h}}
\newcommand{\bfW}{\mathbf{W}}

\newcommand{\cov}{\text{cov}}
\newcommand{\var}{\text{var}}

\begin{document}
\title{Spatial localization for nonlinear dynamical stochastic models for excitable media}
\author{Nan Chen\textsuperscript{(1)}, Andrew J. Majda\textsuperscript{(2)} and Xin T. Tong\textsuperscript{(3)}}
\maketitle
\noindent (1)Department of Mathematics, University of Wisconsin -
Madison, WI, USA

\noindent (2)Department of Mathematics and Center for Atmosphere Ocean Science, Courant Institute of Mathematical Sciences, New York University, New York, NY, USA\\ Center for Prototype Climate Modeling, New York University Abu Dhabi, Saadiyat Island, Abu Dhabi, UAE

\noindent (3)Department of Mathematics, National University of Singapore, Singapore 119076 

\abstract{Nonlinear dynamical stochastic models are ubiquitous in different areas. Their statistical properties are often of great interest, but are also very challenging to compute. Many excitable media models belong to such types of complex systems with large state dimensions and the associated covariance matrices have localized structures. In this article, a mathematical framework to understand the spatial localization  for a large class of stochastically coupled nonlinear systems in high dimensions is developed. Rigorous mathematical analysis shows that the local effect from the diffusion results in an exponential decay of the components in the covariance matrix as a function of the distance while the global effect due to the mean field interaction synchronizes different components and contributes to a global covariance. The analysis is based on a comparison with an appropriate linear surrogate model, of which the covariance propagation can be computed explicitly. Two important applications of these theoretical results are discussed. They are the spatial averaging strategy for efficiently sampling the covariance matrix and the localization technique in data assimilation. Test examples of a  linear model and a stochastically coupled FitzHugh-Nagumo model for excitable media are adopted to validate the theoretical results. The latter is also used for a systematical study of the spatial averaging strategy in efficiently sampling the covariance matrix in different dynamical regimes.}
% It is shown that the spatial averaging strategy is always skillful for short time. The long-term consistency of the spatial averaging strategy requires that the long-term covariance matrix has localized structures.
\section{Introduction}
Nonlinear dynamical stochastic models are ubiquitous in different areas \cite{majda2016introduction, majda2006nonlinear, stelling2001book, sheard2009principles}. Key features of these complex systems are multiscale dynamics, high-dimensional phase space, nonlinear energy transfers, highly non-Gaussian probability density functions (PDFs), intermittent instability, random internal and external forcing as well as extreme events. Understanding and predicting these complex turbulent dynamical systems have both scientific and practical significance. Many excitable media models belong to such types of nonlinear dynamical stochastic systems \cite{lindner2004effects, tyson1988singular, perc2005spatial, holden2013nonlinear}. These excitable media models are ubiquitous in nature and can be found in physical, chemical, and biological systems that are far from thermodynamic equilibrium. One important feature of these excitable medium is that they are susceptible to finite perturbations, where a variety of wave patterns can be triggered. Examples of excitable medium include solitary waves, target like patterns, and spiral waves.

%Modern day scientific computing often involves simulations of complex nonlinear models with refined scale dynamics. The model dimension is usually high, either because the system consists of many particles, or a dense mesh grid  is required to resolve the small scale features within a large domain. For example, in neural science, networks of thousands of neural cells are simulated to understand their collective behavior. In numerical weather prediction (NWP), atmospheric and oceanic simulations are made to forecast outcomes, and because earth is a huge domain, a refined scale  model  often consists of $N=O(10^8)$ components.

For nonlinear dynamical stochastic models, the high dimensionality imposes great challenges to the understanding of the underlying mechanisms, the design of effective prediction schemes and various uncertainty quantification problems \cite{bellman2015adaptive, powell2007approximate}. In particular, the high dimensionality leads to both a large computational cost and an enormous storage even for one single model simulation. Moreover, it is often of interest to  calculate  the time evolution of the statistics, especially the covariance matrix that includes the information of the interactions between different components \cite{anderson1979optimal, kalman1960new}. Since analytical solutions for complex nonlinear systems are typically unavailable, Monte Carlo methods are often applied. However, the affordable sample size is very limited due to the computational cost while the size of the covariance matrix, which is the square of the state variables, can be many magnitudes larger in many real applications that involve millions or billions of state variables. Furthermore, non-Gaussian features due to the nonlinear nature of the underlying dynamics such as extreme events can greatly affect the calculation of the covariance matrix and introduce extra difficulties. Therefore, understanding the covariance structures in nonlinear dynamical stochastic models and then developing efficient computing strategies that can exploit these structures become extremely important.

%High dimensionality imposes severe challenges to uncertainty quantification (UQ). First of all, the computation and storage cost of each model simulation becomes very expensive when the dimension is high, so the affordable sample size is limited. Second, the number of possible inter-component dependence  grow combinatorially fast with the number of components, which makes it difficult to quantify all the component dependences using the data or simulation available. For example, in NWP, the state-of-art computing resources can provide around one hundred samples within the time requirement of weather forecasting. But in order to incorporate satellite data, a covariance matrix of $\frac12N(N+1)=O(10^{16})$ entries needs to be constructed for the representation of the underlying uncertainty.

To facilitate the calculation of the covariance matrix, various approximations are applied in practice. One widely used strategy is to adopt Gaussian approximations for the underlying dynamics since there are many well-developed analytical and numerical methods for sampling high-dimensional Gaussian distributions. However, completely ignoring the non-Gaussian information in the underlying dynamics often leads to large model errors. One effective strategy is to decompose the entire phase space into a high dimensional Gaussian part and a low dimensional non-Gaussian part using the conditional Gaussian framework \cite{CMT14, CMT14b, CM18, CMT17, majda2014blended, lee2015multiscale}. Another strategy is to exploit the fact that the majority of the problem uncertainty is distributed within a low dimensional subspace of the Karhuen-Lo\`{e}ve expansion, then uncertainty quantification (UQ) can be operated just in this subspace for sampling and computational simplicity \cite{SpantiniEtAl15, CuiEtAl14,Cotter13,BuiEtAl13,FlathEtAl11,PetraEtAl14,CuiEtAl16b}.

%To resolve these issues, the general wisdom is to exploit certain statistical structures within the problem. Since high dimensional  direct sampling is feasible for Gaussian distributions, one strategy to approximate the underlying distribution as a Gaussian, or decompose the problem into a high dimensional Gaussian part and a low dimensional non-Gaussian part using the conditional Gaussian framework \cite{CMT14, CMT14b, CM18, CMT17}. Another strategy is exploiting the fact that the majority of the problem uncertainty is distributed within a low dimensional subspace of the Karhuen-Lo\`{e}ve expansion, then UQ can be operated just in this subspace for sampling and computational simplicity \cite{SpantiniEtAl15, CuiEtAl14,Cotter13,BuiEtAl13,FlathEtAl11,PetraEtAl14,CuiEtAl16b}.

This paper focuses on a different structure-exploiting strategy known as localization, which is widely used in many areas including the numerical weather prediction \cite{HoutMitch2001,Hamilletal2001,Houtekamer2005} and efficiently quantifying the uncertainty in excitable media \cite{biktasheva2006localization, ohta1989higher}. In many applications, the system state variables  represent the information at different spatial locations, and the dependence of certain information between different grid points decays quickly when their distance becomes large. In other words, two grid points can be regarded as nearly independent if they are far from each other. The localization strategy exploits this feature and adopts a band matrix as the approximation of the covariance \cite{rekleitis2004particle, anderson2007exploring}. Thus, the localization allows the development of much simpler approaches in analyzing the covariance matrix and thus facilitates the study of many important UQ problems.

%This paper focuses on a different structure-exploiting strategy known as localization, which is routinely used in NWP \cite{HoutMitch2001,Hamilletal2001,Houtekamer2005}. The basic idea is that, in many high dimensional models, the model components represent information of different spatial locations or particle. Empirical study and synthetic data have both shown that while model components depend on each other, such dependence often decays with the spatial distance in between. In other words, for faraway components, the underlying uncertainties are mostly independent. This feature can be exploited for different UQ purposes. When the model is spatially homogenous, faraway components in one single simulation can be viewed as independent samples of the marginal density, therefore the spatial average can be used to approximate the statistical average \cite{CM17pnas}. And to estimate the underlying covariance matrix in NWP, one can judiciously ignore the sample correlation between faraway locations \cite{HoutMitch2001,Hamilletal2001,Houtekamer2005,Tong18}. Recent results have also adapted this idea for inverse problems \cite{MTM18}. Sections \ref{sec:spaceavg} and \ref{sec:localDA} in below will explain why these methods can improve efficiency and accuracy.

While the localization strategy  has a wide range of real-world applications, it remains unclear about the role of different dynamical components  on affecting the covariance localization. In addition, despite many empirical evidences, there is no rigorous justifiction of the covariance decay as a function of the spatial distance. Nevertheless, these are crucial topics in practice for at least two reasons. First, it is known that applying localization by approximating the full covariance matrix with a band matrix may introduce some biases in quantifying the uncertainty of the underlying system \cite{Tong18, nerger15} and the current bias estimates \cite{Tong18,MTM18} require the knowledge of the underlying covariance matrix as well as a quantitative description of the decay in the covariance. Second, in order to apply localization in practice, one often needs to specify a localization radius $L$, which represents the typical distance for two components to be roughly independent. Unfortunately, in the absence of a rigorous theory to understand the covariance decay behavior, determining $L$ often relies on exhaustive tuning and data fitting, which can be expensive and inaccurate.

%While the localization strategy  has a wide range of applications, when and  how does the component dependence decays with spatial distance is not well understood. This issue is important in practice for two reasons. First, it is known that applying localization may introduce a bias in UQ, \cite{Tong18, nerger15} and the current bias estimates \cite{Tong18,MTM18} require the knowledge of the underlying covariance matrix, and how does the covariance decay. Second, to apply localization in practice, one often needs to specify a localization radius $L$. It represents the typical distance for two components to be roughly independent. With no rigorous understanding of the covariance decay phenomena, finding $L$ often relies on exhaustive tuning and data fitting, which can be expensive and inaccurate.

This paper aims at developing a theoretical framework to understand the covariance behavior for a large class of stochastically coupled nonlinear systems in high dimensions that include many excitable media models and other nonlinear complex models.
Rigorous mathematical analysis shows that the local effect from the diffusion results in an exponential decay of the components in the covariance matrix as a function of the distance while the global effect due to the mean field interaction synchronizes different components and contributes to a global covariance. To achieve this, a covariance comparison principle is established. Then by explicitly computing the covariance propagation in a linear surrogate model, an upper bound on the covariance entries can be obtained.

Two important applications of these theoretical results are illustrated. First, a spatial averaging sampling strategy is developed. It makes use of the statistical symmetry to replace the samples from repeated experiments using ensemble methods by samples at different spatial grid points \cite{CM17pnas, ricketson2015multilevel}, which greatly reduces the number of repeated experiments while the accuracy remains the same. In other words, combining the spatial averaging strategy with a small number of Monte Carlo simulation of the underlying nonlinear complex systems, the covariance with a localization structure can be sampled in an efficient and accurate fashion, at least for a finite time. Secondly, the same structure can be exploited to improve the accuracy of data assimilation \cite{kalnay2003atmospheric, majda2012filtering, evensen2009data, law2015data}. The theory developed here provides a practical guideline to choose the localization radius $L$ in a systematical way that facilitates an effective data assimilation for large dimensional systems.

The rest of the article is organized as follows. Section \ref{Sec:MainResults} states the main theoretical results for the spatial localization for nonlinear dynamical stochastic models and the two applications. Section \ref{Sec:Analysis} includes the detailed analysis of the covariance matrix with both local and nonlocal effects. Numerical simulations of  a  linear model and a stochastically coupled FitzHugh-Nagumo (FHN) model for excitable media in different dynamical regimes are shown in Section \ref{Sec:Numerics}. The article is concluded in Section \ref{Sec:Conclusions}. Details of proofs are included in Appendix.
%This paper intends to amend this theoretical gap by building a theoretical framework which explicit shows the covariance decay structure within certain high dimensional models. This framework can apply to discrete stochastic diffusion reaction models with  mean-field interactions. One motivating example is the FitzHugh-Nagumo (FHN) model \cite{lindner2004effects}, which is used in neural science to study pattern formulation. We limit most of our discussion in a fixed time regime, while Section \ref{sec:longtime} discusses the possibility in extending the results to  long time settings.

\section{Problem setup and main results}\label{Sec:MainResults}
To describe the behavior of nonlinear dynamical stochastic models for an excitable media, we consider a general  stochastic diffusion reaction model with mean field interaction on a one dimensional torus $[0,1]$:
\begin{equation}
\label{eqn:reactiondiffusion}
\frac{d}{dt} u_t=\nu\Delta u_t+ F(u_t)+\int h(u_t(x))dx+\dot{W}(t,x),
\end{equation}
where $W(t,x)$ is  a spatial white noise. We can discretize \eqref{eqn:reactiondiffusion} in space, and write down the dynamics at the $i$-th grid point $u_i(t)=u_t(\frac{i}{N})$ as
\begin{equation}
\label{eqn:discretize}
du_i(t)= \frac{\nu}{2}N^2 (u_{i+1}(t)+u_{i-1}(t)-2u_i(t))dt+F(u_i(t))dt+\frac1N\sum_{j=1}^N h(u_j(t))+dw_i(t),
\end{equation}
where $w_i(t)$ are independent standard Wiener processes.

The deterministic forcings in \eqref{eqn:discretize} can be classified into two groups based on their range of action. Here $\frac{\nu}{2}N^2 (u_{i+1}(t)+u_{i-1}(t)-2u_i(t))+F(u_i(t))$ describes how does a component $u_i$ force its neighbors and itself while $\frac1N\sum_{j=1}^N h(u_j(t))$ models the McKean Vlasov type of mean field interaction \cite{mckean1966class}. We will adapt this classification to simplify our notation. Moreover, we will consider the general setting where $u_i$ is not univariate. This will be useful when we discuss the FHN model, as we need both the membrane voltage and the recovery forcing.

Specifically, we let $\bfx(t)=(\bfx_1(t),\cdots, \bfx_N(t))$ be a stochastic process in $\reals^{d}$. For simplicity, we assume all block share the same dimension, that is $\bfx_i(t)\in \reals^q$ and $d=qN$. Suppose each block follows an SDE
\begin{equation}
\label{sys:cyclicSDE}
d\bfx_i =\bff (t, \bfx_{i-1}, \bfx_{i}, \bfx_{i+1}) dt+\frac1N\sum_{j=1}^N\bfh(t,\bfx_j) dt +\Sigma d \bfw_i(t),\quad i=1,\ldots, N.
\end{equation}
In \eqref{sys:cyclicSDE}, $\bfw_i(t)$ are independent standard Wiener processes of dimension $q$, and $\Sigma$ is symmetric. The index $i$ should be interpreted in a cyclic fashion,  that is $\bfx_0=\bfx_N, \bfx_{1}=\bfx_{N+1}$. The natural distance   between indices is given by
\[
\bfd(i,j)=\min\{|i-j|, |i+N-j|, |j-i+N|\}.
\]
For simplicity, we assume $\{\bfx_i(0)\}_{i=1,\ldots,N}$ are i.i.d. samples from  $\mathcal{N}(m_0,\Sigma_0 \Sigma^T_0)$.

Many high dimensional stochastic models can be formulated as \eqref{sys:cyclicSDE}. Examples include the Lorenz 96 model \cite{lorenz1996predictability} and the FHN model \cite{lindner2004effects}.

Our main results can be summarized as the follows.
\begin{thm}
\label{thm:rough}
Suppose the following constants are independent of $N$
\begin{equation}
\label{eqn:simhomo}
\begin{gathered}
\lambda_0:= \sup_{i,\bfx,s\leq t}\{\lambda_{max}\left(\nabla_{\bfx_i}\bff(s,\bfx_{i-1}, \bfx_{i}, \bfx_{i+1})+\nabla_{\bfx_i}\bff(s,\bfx_{i-1}, \bfx_{i}, \bfx_{i+1})^T\right)\},\\
\lambda_F:= \sup_{i,\bfx,s\leq t}\{\|\nabla_{\bfx_{i-1}}\bff(s,\bfx_{i-1}, \bfx_{i}, \bfx_{i+1})\|,\|\nabla_{\bfx_{i+1}}\bff(s,\bfx_{i+1}, \bfx_{i}, \bfx_{i+1})\|\},\\
\lambda_H:=\sup_{i, \bfx_i, s\leq t}\|\nabla_{\bfx_i} \bfh(s,\bfx_i)\|.
\end{gathered}
\end{equation}
Here $\lambda_{max}$ denotes the maximum eigenvalue. Then for any test function $g$ and $\beta>0$, there is a constant $C_{\beta,t}$ independent on $N$, such that
\[
\cov(g(\bfx_i(t)), g(\bfx_j(t)))\leq C_{\beta,t}\|\nabla g\|^2_\infty  \left(e^{-\beta \bfd(i,j)} + \frac{1}{N} \right).
\]
Here $\|\nabla g\|_\infty:=\sup_\bfx \|\nabla g(\bfx)\|$. And if $\bfh\equiv 0$, the bound can be improved to $C_{\beta,t}\|\nabla g\|^2_\infty  e^{-\beta \bfd(i,j)}$. \end{thm}
The exact value of $C_{\beta,t}$ can be found below in Theorem \ref{thm:heat}.
The main conclusion here is that when the   dimension $N$ is large,  for components that are far from each other,
their covariance is close to zero. In other words, the covariance between model components is significant only for nearby components.
Following \cite{MTM18}, we call such a covariance structure to be \emph{local}. Below, we illustrate two important applications that show such a property is crucial for uncertainty quantification (UQ).

\subsection{Sampling with spatial averaging}\label{Sec:SpatialAveraging}
\label{sec:spaceavg}
One fundamental  UQ question regarding SDEs such as \eqref{sys:cyclicSDE} is how to efficiently compute $\E g(\bfx_i(t))$ for a given test function and ensure that the estimation error is small. The standard Monte Carlo approach involves simulating \eqref{sys:cyclicSDE} multiple times, and use the sample average $\hat{g}_{mc}=\frac1K \sum_{j=1}^K g(\bfx^{(j)}_i(t))$ as an estimator. However, with the increase of the dimension $N$, the computational costs to generate each independent  sample $\bfx^{(j)}(t)$ increases dramatically. Therefore, only a small number of samples $K$ can be adopted in practice.

One important observation is that \eqref{sys:cyclicSDE} is invariant under index shifting. Therefore, by symmetry, $\bfx_i(t)$ shares a common distribution for all $i$. As a consequence, each of them can be viewed as an individual, despite not independent, sample. Exploiting this perspective, it has been proposed in \cite{CM17pnas} the following spatial average as an estimator of $\E g(\bfx_i(t))$
 \[
 \hat{g}_{K,N}=\frac1{K}\sum_{j=1}^K\frac 1N \sum_{i=1}^N g(\bfx^{(j)}_i(t)).
 \]
Since the samples are independent, the statistical properties of $\hat{g}_{M,N}$ is determined by the ones of
\begin{equation}
\label{eqn:ghat}
\hat{g}_N=\frac1N \sum_{i=1}^N g(\bfx_i(t)).
\end{equation}
It is straightforward to verify that  $\hat{g}_N$ and $\hat{g}_{K,N}$ are unbiased:
 \[
 \E\hat{g}_{K,N}= \E\hat{g}_N=\frac1N \sum_{i=1}^N \E g(\bfx_i(t))=\E g(\bfx_1(t)).
 \]
 The accuracy of  these estimators can be measured by their variance. By independence, we have $\var \hat{g}_{K,N}=\frac1K  \hat{g}_N$, which is given by
\begin{equation}
\label{eqn:varsum}
\var \hat{g}_N=\frac{1}{N^2} \sum_{i,j=1}^N \cov(g(\bfx_i), g(\bfx_j)) =\frac1N \sum_{j=1}^N  \cov(g(\bfx_1), g(\bfx_j)).
\end{equation}
Then plugging in the estimate from Theorem \ref{thm:rough}, the result in \eqref{eqn:varsum} can be bounded by
\[
\var \hat{g}_N\leq
\frac1N \sum_{j=1}^N C_{\beta,t}
\|\nabla g\|^2_\infty  \left(e^{-\beta \bfd(1,j)} + \frac{1}{N} \right)\leq \frac1N \frac{C_{\beta,t } \|\nabla g\|^2_\infty }{1-e^{-\beta}}.
\]
Therefore we have
\[
\var \hat{g}_{K,N}\leq \frac1{NK} \frac{C_{\beta,t } \|\nabla g\|^2_\infty }{1-e^{-\beta}}.
\]
In contrast, if we consider the standard estimator $\hat{g}_{mc}$, its variance is $\frac1K \var(g(\bfx_i(t)))$, which is of $O(N)$ multiple of the one of  $\var \hat{g}_{M,N}$. In other words, the spatial averaging has increase the effective sample size by a factor of $N$.

\subsection{Localization in  data assimilation}
\label{sec:localDA}
In data assimilation (DA), sequential partial observations are incorporated to improve predictions of a stochastic dynamical system  \cite{kalnay2003atmospheric, majda2012filtering, evensen2009data, law2015data}. One major application of DA is numerical weather prediction, where forecasts are obtained by combining simulations of the dynamical models with satellite data. The associated dimensions of such data assimilation problems are extremely high, often in orders of millions in order to describe variables in different temporal and spatial scales of the underlying nonlinear complex systems.
Moreover, for these refined models, the current computational power can only provide around one hundred simulation samples. DA algorithms, for example the ensemble Kalman filter (EnKF), need to use  these 100 samples to represent the covariance matrix of model components.

The classic sample covariance,
\[
\widehat{C}:=\frac1{K-1}\sum_{j=1}^K (\bfx^{(j)}(t)-\bar{\bfx}) \cdot (\bfx^{(j)}(t)-\bar{\bfx}(t))^T, \quad \bar{\bfx}(t)=\frac1K\sum_{j=1}^K \bfx^{(j)}_t.
\] is not accurate in this setting.
A classic random matrix results \cite{BY88} indicates that if $\bfx^{(i)}_t$ are i.i.d. samples from $\mathcal{N}(0, C)$, the covariance estimation error $\|C-\widehat{C}\|=O(\sqrt{N/K})$. Therefore, the estimator is very inaccurate in the situations with the small samples and high dimensions, i.e., $K\ll N$. The intuition behind this mathematical phenomenon is that, spurious correlation may exist between each pair of components with a small probability because of the sampling effect, but there are $O(N^2)$ pairs of components. Thus, the resulting errors are very likely to be significant.

Nevertheless, if we are aware that covariance is local as described in Theorem \ref{thm:rough}, then the erroneous estimation can be made accurate. Note that the far-off-diagonal entries are close to zero. Thus, a natural way is to truncate these entries. In particular, with a given bandwidth $L$ and covariance matrix $C$, we define the localization of $C$ as
\[
[C^L]_{i,j}=[C]_{i,j} \mathbf{1}_{\bfd(i,j)\leq L}.
\]
It has been shown in \cite{BL08} that such a localization technique can significantly improve the sampling accuracy. A simplified version in \cite{Tong18} shows that if $\bfx(t)$ is Gaussian distributed, then there exists a uniform constant  $c$, so that for any $\epsilon>0$,
\begin{equation}
\label{tmp:RMT}
\Prob(\|\widehat{C}^L-C^L\|> \epsilon)\leq 8 \exp\left(2\log N- cK \min\left\{\frac{\epsilon}{2L \|C\|},\frac{\epsilon^2}{4L^2 \|C\|^2}\right\}\right).
\end{equation}
In other words, if $L$ is a fixed bandwidth, it requires only $L^2 O(\log N)$ many samples to find an accurate estimate of $C^L$.

On the other hand, one needs to notice that \eqref{tmp:RMT} indicates  $\widehat{C}^L$ is a good estimator of $C^L$, but not necessarily $C$. In order for it to also be  a good estimator of $C$, one  needs $\|C-C^L\|$ to be small as well. Theorem \ref{thm:rough} can guarantee this if there is only local interaction, i.e. $\bfh\equiv 0$. Since a covariance matrix's $l_2$-norm is bounded by its $l_1$-norm, we can plug in the estimate in Theorem \ref{thm:rough} with $g(x)=x$ and find
\begin{equation}
\label{eqn:localerror}
\|C-C^L\|\leq \max_i \sum_{\bfd(i,j)>L} \left|[C]_{i,j}\right|\leq \frac{2C_{\beta,t} e^{-\beta L}}{1-e^{-\beta}}.
\end{equation}
In practice, given an error threshold $\epsilon>0$, one can combine \eqref{eqn:localerror} and \eqref{tmp:RMT} together to find a bandwidth $L$ and sample size $K=O(L^2 \log N)$, such that
the covariance estimation error $\|\widehat{C}^L-C\|$ is below $2\epsilon$ with high probability.

Since localization involves using very simple numerical operation to significantly improve the estimation of the covariance matrix, it has been widely applied in all DA algorithms, in particular, the ensemble Kalman filter \cite{evensen03, HM98, HWS01, WH02, MY07, HKS07, TMK15, TMK15non}. Moreover, adaptation of the localization strategy can be applied to Bayesian inverse problems where the associated covariance matrices follow the description of Theorem \ref{thm:rough}. In particular, it has been shown in \cite{MTM18} that blocking the high dimensional components, and then applying Gibbs sampler, the resulting Markov chain Monte Carlo algorithms will converge to the desired distribution with a rate independent of the dimension. One fundamental assumption of all these practical computational methods is that the underlying system is spatially localized, yet there has been no rigorous justification of this assumption. This article closes this gap for systems of type of \eqref{sys:cyclicSDE}.

\subsection{Preliminaries}

 To facilitate our discussion, we will use the following notations: $\|v\|$  denotes the $l_2$ norm of $v$, while $\langle u, v\rangle=u^T v$ is the inner product; $\bfI_q$ denotes the identity matrix of dimension $q$. Given a matrix $A$, the norm $\|A\|$ is the $l_2$ operator norm of $A$; $\lambda_{max}(A)$ denotes the maximum eigenvalue of $A$ and $\|A\|_F$ is its Frobenius norm.  We use $[A]_{i,j}$ to denote the $(i,j)$-entry of $A$.

When we have a vector $\bfx$ of $qN$ dimension, we often view it as the concatenation of $N$ sub-blocks of dimension $q$. The decomposition can be written as $\bfx=(\bfx_1,\cdots, \bfx_N)$. Given a $qN\times qN$ matrix $A$, we write its $(i,j)$-th $q\times q$ sub-block as $A_{i,j}$ or $\{A\}_{i,j}$. The readers should not confuse these notations with the matrix entries $[A]_{i,j}$. In particular, we can write down the entries of $A_{i,j}$:
\[
[A_{i,j}]_{m,n}=[\{A\}_{i,j}]_{m,n}=[A]_{(i-1)q+m,(j-1)q+n},\quad i,j=1,\cdots, N,\quad m,n=1,\cdots,q.
\]
With a function $f$,  $\nabla f$ denotes its Jacobian matrix. $\nabla_{\bfx_i} f(\bfx)$ denotes the partial derivative with respect to $\bfx_i$, which is the $i$-th block-column of $\nabla f$.

For two symmetric matrices $A\preceq B$ indicates that $B-A$ is positive semidefinite.

\section{Analysis of covariance matrices}\label{Sec:Analysis}
We provide the explicit statements of our result in this section. The detailed proofs are allocated in Appendix.

\subsection{Linear covariance propagation}
%As a matter of fact, the spatial homogeneity is not necessary the covariance structure to be local. We will consider the following model:
Consider the following multivariate SDE:
\begin{equation}
\label{sys:SDE}
d\bfx_i(t)=\bff_i(t,\bfx_{i-1}, \bfx_{i}, \bfx_{i+1}) dt +\frac1N\sum_{j=1}^N\bfh_i(t,\bfx_j) dt+\Sigma d \bfw_i(t),\quad i=1,\ldots, N.
\end{equation}
Our main result is established through a covariance comparison principle between \eqref{sys:SDE} and a linear surrogate.
\begin{thm}
\label{thm:main}
Suppose the covariance propagation of \eqref{sys:SDE} is dominated by $(F_*(t), H_*(t))\in (\reals^{N\times N},  \reals^{N\times N})$ in the following sense:
\begin{enumerate}[(1)]
\item $\nabla_{\bfx_i}\bff_i(t,\bfx_{i-1}, \bfx_{i}, \bfx_{i+1})+\nabla_{\bfx_i}\bff_i(t,\bfx_{i-1}, \bfx_{i}, \bfx_{i+1})^T \preceq 2[F_*(t)]_{i,i} \bfI$.
\item $\|\nabla_{\bfx_{i-1}}\bff_i(t,\bfx_{i-1}, \bfx_{i}, \bfx_{i+1})\|\leq [F_*(t)]_{i,i-1}, \|\nabla_{\bfx_{i+1}}\bff_i(t,\bfx_{i-1}, \bfx_{i}, \bfx_{i+1})\|\leq [F_*(t)]_{i,i+1}$.
\item $\frac1N \|\nabla_{\bfx_j} \bfh_i (t,\bfx_j)\|\leq [H_{*}(t)]_{i,j}$.
\end{enumerate}
Let $G_*(s)=F_*(s)+H_*(s)$, and
\[
Q(s)=  \exp\left(\int^t_{t-s} G_* (r) dr\right) \exp\left(\int^t_{t-s} G_* (r) dr\right)^T.
\]
Then for any test function $g$ with bounded gradient,
\begin{equation}
\label{eqn:covspace}
|\text{cov} (g(\bfx_i (t)), g(\bfx_j(t)))|\leq \sqrt{q} \|\nabla g\|^2_\infty\left(\|\Sigma^2_0\|_F [Q(t)]_{i,j}+\|\Sigma^2\|_F\int^t_0 [Q(s)]_{i,j} ds\right).
\end{equation}
\end{thm}
This theorem shows that if the covariance propagation of \eqref{sys:SDE} is dominated by $G_*=F_*+H_*$, then the spatial covariance of \eqref{sys:SDE} is also dominated by the  linear surrogate model associated with $G_*$:
\begin{equation}
\label{sys:linearSDE}
d\bfx_i=[G_*]_{i,i}\bfx_i dt+[G_*]_{i,i-1} \bfx_{i-1}dt+ [G_*]_{i,i+1}\bfx_{i+1}dt+\Sigma d \bfw_i(t),\quad i=1,\ldots, N.
\end{equation}

To see that \eqref{sys:linearSDE} has the spatial covariance structure as in \eqref{eqn:covspace}, we concatenate \eqref{sys:linearSDE} into one equation for $\bfx(t)$
\[
d\bfx(t)=D_{G_*}\bfx dt+D_{\Sigma} d \bfw(t),\quad \bfx(0)\sim \mathcal{N}(0,D_{\Sigma_0}).
\]
Here, $D_{G_*}$,$D_{\Sigma}$ and $D_{\Sigma_0}$ are $qN\times qN$ dimensional matrices. Their $q\times q$ sub-blocks are given by
\[
\{D_{G_*}\}_{i,j}=[G_*]_{i,j}\bfI_q,\quad \{D_{\Sigma}\}_{i,j}=\unit_{i=j} \Sigma,\quad  \{D_{\Sigma_0}\}_{i,j}=\unit_{i=j} \Sigma_0.
\]
Since this equation is linear, we can apply the Duhamel's formula and write the solution as
\[
\bfx(t)=q(t)\bfx(0)+ \int^t_0  q(s)  D_{\Sigma} d\bfw(s),\quad q(s):=\exp\left(\int^t_{t-s} D_{G_*} (r) dr\right).
\]
By Ito's isometry,
\[
\cov\, \bfx(t)= q(t)D_{\Sigma_0}D_{\Sigma_0}^T  q(t)^T+ \int^t_0 q(s)D_\Sigma D_\Sigma^T   q(s)^Tds.
\]
In particular, the covariance between $\bfx_i(t)$ and $\bfx_j(t)$ is given by $\{\cov\, \bfx(t)\}_{i,j}$. Since
$D_{\Sigma_0}$ and $D_\Sigma$ are block diagonal, it is not difficult to establish \eqref{eqn:covspace} for system \eqref{sys:linearSDE}.

\subsection{A representation of the covariance matrix}
One difficulty in reaching Theorem \ref{thm:main} is how to compute the covariance matrix for a multivariate SDE.
We will establish a framework for covariance matrix computation, starting from a simple Stein's lemma type of formula:
\begin{lem}
\label{lem:Gaussint}
Suppose $X,Y$ are independent standard Gaussian random variables of dimension $d$, $f$ and $g$ are  smooth functions with bounded first order derivatives. Let  $X^\theta=\cos\theta\, X +\sin\theta\, Y$, then
\[
\text{cov}(f(X),g(X))=\E f(X)g(X)-f(X)g(Y)=\int^{\frac\pi2}_0 \sin \theta  \E \langle \nabla f(X), \nabla g(X^\theta)\rangle d\theta.
\]
\end{lem}
In order to apply Lemma \ref{lem:Gaussint}, we need to represent $\bfx(t)$ as a function of certain Gaussian random variables. These Gaussian random variables are the increments of the Wiener processes driving $\bfx(t)$. To setup this connection, recall that one way to simulate \eqref{sys:SDE} is to perturb its ordinary differential equation (ODE) version with random noises. In particular, consider the ODE system
\begin{equation}
\label{sys:ODE}
d\bfx_i(t)=\bff_i(t,\bfx_{i-1}, \bfx_{i}, \bfx_{i+1}) dt+\frac1N\sum_{j=1}^N\bfh_i(t,\bfx_j) dt ,\quad i=1,\ldots, N.
\end{equation}
 Let $\Psi_h:\reals^{qN}\mapsto \reals^{qN}$ be the mapping from $\bfx(t)$ to $\bfx(t+h)$. We suppress the dependence of $\Psi_h$ on $t$, for notational simplicity.   Then one way to simulate \eqref{sys:SDE} is generate the following iterates
\begin{equation}
\label{eqn:update}
\bfX(0)=m_0+\Sigma_0 \bfW_0,\quad \bfX(n+1)=\Psi_h(\bfX(n))+\Sigma_h \bfW_n,\quad \Sigma_h=\sqrt{h}D_\Sigma.
\end{equation}
Then $\bfX(T)$ is a numerical approximation of $\bfx(t)$ if $Th=t$ when $T$ is large. A standard verification Lemma \label{lem:htozero} of this statement is provided in Appendix. We write the $i$-th block of $\bfX(n)$ as $\bfX_i(n)\in \reals^q$. $\bfW_n$ are $qN$ dimensional standard Gaussian noise for $n=0,1,\cdots,T$.

The strategy here is to consider $\bfX(T)$ instead of $\bfx(t)$, because its dependence on the noise realization $\bfW_i$ is easier to handle using Lemma \ref{lem:Gaussint}. We can show a similar result for $\bfX(T)$ as in Theorem \ref{thm:main}. Then in the limit $h\to 0$, we have our desired results. The details are shown in Appendix. As a remark, it might be possible and more elegant to represent the covariance of $\bfx_t$ directly using tools from Malliavin calculus \cite{Nua95}, and avoid the procedures of discretization and taking continuous limit. Yet the current proof strategy can be easier to understand for an applied math audience.

\subsection{Diffusion and mean field interaction}
While \eqref{eqn:covspace} provides a general upper bound, from its formulation is not easy to see the covariance decay we expect. To make the covariance decay explicit, we try to use instead a homogenous linear surrogate. We update the constant definition in \eqref{eqn:simhomo} to the spatial inhomogeneous setting:
\begin{equation}
\label{eqn:homo}
\begin{gathered}
\lambda_0:= \sup_{i,\bfx,s\leq t}\{\lambda_{max}\left(\nabla_{\bfx_i}\bff_i(s,\bfx_{i-1}, \bfx_{i}, \bfx_{i+1})+\nabla_{\bfx_i}\bff_i(s,\bfx_{i-1}, \bfx_{i}, \bfx_{i+1})^T\right)\},\\
\lambda_F:= \sup_{i,\bfx,s\leq t}\{\|\nabla_{\bfx_{i-1}}\bff_i(s,\bfx_{i-1}, \bfx_{i}, \bfx_{i+1})\|,\|\nabla_{\bfx_{i+1}}\bff_i(s,\bfx_{i-1}, \bfx_{i}, \bfx_{i+1})\|\},\\
\lambda_H:=\sup_{i,j, \bfx, s\leq t}\|\nabla_{\bfx_j} \bfh_i (s,\bfx_j)\|.
\end{gathered}
\end{equation}
They are finite under the conditions of Theorem \ref{thm:main}. We can build a homogeneous  linear surrogate by letting
\[
[F_*]_{i,i}=\lambda_0,\quad [F_*]_{i,i\pm 1}=\lambda_F,\quad [H_*]_{i,j}=\frac1N\lambda_H.
\]
Then the linear surrogate \eqref{sys:linearSDE} corresponds
\begin{equation}
\label{eqn:linearsurro}
d\bfx_i=(\lambda_0+2\lambda_F) \bfx_i dt+\lambda_F \Delta \bfx_i dt+\frac{\lambda_H}N \sum_{j=1}^N \bfx_j +\Sigma d \bfw_i(t),\quad i=1,\ldots, N,
\end{equation}
where the discrete Laplacian is given by
\[
\Delta \bfx_i:=\bfx_{i-1}+\bfx_{i+1}-2\bfx_i,
\]
which is closely related to the heat equation.

The spatial covariance of $\bfx(t)$ is built up by the diffusion effect of $\Delta \bfx_i$ and the mixing effect of mean field interaction $\frac1N \sum_{j=1}^N \bfx_j$. These two effects operate in different ways. The diffusion propagates the  diagonal entries in the covariance matrix to the off-diagonal ones, it contributes to the exponential decay term $e^{-\beta \bfd(i,j)}$ in Theorem \ref{thm:rough}. The mean field interaction synchronizes the components, and contributes to a global covariance of scale $\frac1N$.

We can establish the following explicit bounds in terms of $\lambda_G, \lambda_F, \lambda_H$, instead of the matrix $Q(t)$:
\begin{thm}
\label{thm:heat}
Under the same conditions of Theorem \ref{thm:main}, for any $\beta>0$, let
\[
\lambda_\beta=\lambda_0 + \lambda_F  (e^\beta+e^{-\beta}),\quad  \eta_{\beta}=\lambda_0 + \lambda_F  (e^\beta+e^{-\beta})+\lambda_H
\]
with $\lambda_0, \lambda_F$ and $\lambda_H$ defined by \eqref{eqn:homo}. Then with any function $g$,
\begin{align*}
|\text{cov} (g(\bfx_i (t)), g(\bfx_j(t)))|&\leq  2\sqrt{q}\|\nabla g\|^2_\infty\left( e^{\lambda_\beta t}\|\Sigma^2_0\|_F + (e^{\lambda_\beta t}-1)\|\Sigma^2\|_F/\lambda_\beta \right) e^{-\beta \bfd(i,j)}\\
&+\frac{2\sqrt{q}(1+e^{-\beta})\|\nabla g\|^2_\infty}{(1-e^{-\beta}) N }\left(\left( e^{\eta_{\beta} t} -e^{\lambda_\beta t}\right)\|\Sigma_0^2\|_F+
\left(\frac{e^{\eta_\beta t}-1}{\eta_\beta}-\frac{e^{\lambda_\beta t}-1}{\lambda_\beta}\right)\|\Sigma^2\|_F\right).
\end{align*}
\end{thm}
When there is only mean field interaction, i.e. $\lambda_F= 0$, the result can be simplified by choosing $\beta\to \infty$,
\begin{equation}\label{Bound1}
\begin{split}
|\text{cov} &(g(\bfx_i (t)), g(\bfx_j(t)))|\\&\leq \frac{2\sqrt{q}\|\nabla g\|^2_\infty}{ N }\left(\left( e^{(\lambda_0+\lambda_H) t} -e^{\lambda_0 t}\right)\|\Sigma_0^2\|_F+
\left(\frac{e^{ (\lambda_0+\lambda_H)t}-1}{\lambda_0+\lambda_H}-\frac{e^{\lambda_0 t}-1}{\lambda_0}\right)\|\Sigma^2\|_F\right).
\end{split}
\end{equation}
When there is only diffusion effect, that is $\bfh_i\equiv 0$, the result can be simplified as
\begin{equation}\label{Bound2}
|\text{cov} (g(\bfx_i (t)), g(\bfx_j(t)))|\leq  2\sqrt{q}\|\nabla g\|^2_\infty\left( e^{\lambda_\beta t}\|\Sigma^2_0\|_F + (e^{\lambda_\beta t}-1)\|\Sigma^2\|_F/\lambda_\beta \right) e^{-\beta \bfd(i,j)}.
\end{equation}
It is worth noticing the result holds for all $\beta>0$. This indicates the covariance structure actually decays faster than exponential. In practice, one can try to minimize the right hand side with respect to $\beta$, or directly compute the covariance of \eqref{eqn:linearsurro} to find the exact  local covariance structure. While sharper theoretical upper bound maybe obtainable by applying formulas as in \cite{Ise00}, they are complicated and work only asymptotically.
Our current upper bound is simpler, and sufficient for applications discussed in Sections \ref{sec:spaceavg} and \ref{sec:localDA}.

With Theorem \ref{thm:heat}, it is straightforward to show the spatial averaging estimator is consistent:
\begin{cor}
\label{cor:sampling}
Under the same settings of Theorem \ref{thm:heat},  the following bound holds for the block estimator \eqref{eqn:ghat}:
\begin{align}
\notag
\var \hat{g}_N&\leq \frac{2\|\nabla g\|^2_\infty}{(1-e^{-\beta}) N}  \left(2 e^{2\lambda_\beta t}\|\Sigma_0^2\|_F + (e^{2\lambda_\beta t}-1)\|\Sigma^2\|_F/\lambda_\beta \right)\\
\label{eqn:estbound}
&\quad+\frac{2\|\nabla g\|^2_\infty}{(1-e^{-\beta}) N }\left(\left( e^{\eta_{\beta} t} -e^{\lambda_\beta t}\right)\|\Sigma_0^2\|_F+
\left(\frac{e^{\eta_\beta t}-1}{\eta_\beta}-\frac{e^{\lambda_\beta t}-1}{\lambda_\beta}\right)\|\Sigma^2\|_F\right).
\end{align}
\end{cor}
\begin{proof}
Since in \eqref{eqn:varsum}, $\sum^N_{j=1} e^{-\beta \bfd(1,j)}\leq 2\sum^{\infty}_{j=0}e^{-\beta j}=\frac{2}{1-e^{-\beta}}$.
\end{proof}

\subsection{Long time stability}
\label{sec:longtime}
In many applications, the long time $t\to \infty$ scenario is of interest. For example, if we want to find the equilibrium measure of $\bfx(t)$, we often simulate $\bfx(t)$ for a long time, and use its distribution as an approximate. Likewise, if we are interested in data assimilation with models as \eqref{sys:cyclicSDE}, the algorithms are usually iterated for an extended period. In order to apply the spatial averaging strategy in Section \ref{sec:spaceavg}, and the localization strategy in Section \ref{sec:localDA} for these operations, we need the local covariance structure to be stable in time.

In the view of Theorem \ref{thm:heat}, we can guarantee the local covariance structure is stable if $\lambda_0+2\lambda_F+\lambda_H<0$. In particular we have the following corollary by letting $t\to \infty$ in the estimate.
\begin{cor}
\label{cor:longtime}
If $\lambda_0+\lambda_H+2\lambda_F<0$, then there is a $\beta_0>0$, so that if $0<\beta<\beta_0$, $\lambda_\beta\leq \eta_\beta<0$ we have
%\[
%\var \hat{f}_N\leq \frac2{\beta N} \|\nabla f\|^2_\infty \left(2\|\Sigma_0\|^2_F + \frac{1-e^{2\lambda_\beta t}}{-\lambda_\beta}\|\Sigma\|^2_F \right).
%\]
\begin{align*}
\limsup_{t\to \infty} |\text{cov} (g(\bfx_i (t)), g(\bfx_j(t)))| &\leq  -\sqrt{q}\|\nabla g\|^2_\infty \|\Sigma^2\|_F\left(\frac{2}{\lambda_\beta} e^{-\beta \bfd(i,j)}-\frac{2(1+e^{-\beta})}{(1-e^{-\beta}) N }
\left(\eta^{-1}_\beta-\lambda^{-1}_\beta\right)\right).
\end{align*}
\end{cor}

%\subsection{Other results? Diffusive Scaling?}
%If we interpret \eqref{sys:cyclicSDE} in terms of the heat equation, one would expect the covariance decays like $\exp(-c \bfd(i,j)^2)$, instead of exponential.
%In fact, this is used in EnKF as the Gaspari-Cohn kernel, also the numerical correlation plot looks Gaussian decay instead of exponential.
%
%This is probability doable in theory, but we need to scale the dimension $N$ with $\sqrt{t}$. That is we consider a sequence of \eqref{sys:cyclicSDE} at ever larger $N$ and $t$. I don't know if this make sense in practice. I think the setting of Corollary \ref{cor:sampling} is more useful, that is, fix a time $t$ and look at ever larger $N$. But we can also consider the scaling limit.

\section{Test Examples}\label{Sec:Numerics}
In this section, we demonstrate our theoretical findings with several numerical test examples.
\subsection{A linear model}
We start with a linear model, where the analytic solutions associated with this linear model can be written down explicitly (See Appendix \ref{app:linear}). Therefore, this model is able to provide insights and validation of the theories developed in the previous sections.

The linear model is as follows:
\begin{equation}\label{linear_model}
  \frac{du_i}{dt} = -a u_i + d_u(u_{i+1}-2u_i + u_{i-1}) + w(\bar{u} - u_i) + \sigma_u\dot{W}_{u_i}, \qquad \mbox{for~} i = 1,\ldots, N,
\end{equation}
with periodic boundary conditions. Here $\bar{u}$ is the averaged value of all the state variables,
\begin{equation*}
  \bar{u} = \frac{1}{N}\sum_{i=1}^N u_i.
\end{equation*}
In the linear model \eqref{linear_model}, the coefficients of damping $a$, diffusion $d_u$, mean field interaction $w$ and the stochastic noise $\sigma_u$ are all constants. Thus, the flow field is statistically homogeneous. Below, we fix
\begin{equation}\label{linear_model_fixed_coefficients}
  a = 1, \qquad\mbox{and}\qquad \sigma_u = 0.5
\end{equation}
while the diffusion and mean field interaction coefficients $d_u$ and $w$ vary in different cases.

In the linear model \eqref{linear_model}, it is easy to write down the constants in \eqref{eqn:homo},
\begin{equation}\label{linear_model_constant1}
  \lambda_0 = -a - 2d_u - w,\qquad \lambda_F = d_u,\qquad\mbox{and}\qquad \lambda_H = w.
\end{equation}
In addition, only one-layer model is utilized here and we consider the covariance of different $u_i$ themselves. Therefore,
\begin{equation}\label{linear_model_constant2}
    q = 1\qquad\mbox{and}\qquad\nabla g = 1.
\end{equation}
Furthermore, $\|\Sigma^2\|_F = \sigma_u^2$. Below, we study the role of the mean field interaction and the diffusion. All the results shown in this subsection are computed using the exact solution given in Appendix \ref{app:linear}, which allows us to understand Theorem \ref{thm:heat} without being interfered by numerical errors. \medskip\medskip

\noindent\textbf{1. Linear model with only mean field interaction.} First, we study the case with only mean field interaction. In other words, we set the diffusion coefficient $d_u = 0$. The constants in \eqref{linear_model_constant1} reduce to
\begin{equation}\label{linear_model_constant1_mean_field}
  \lambda_0 = -a  - w,\qquad \lambda_F = 0,\qquad\mbox{and}\qquad \lambda_H = w.
\end{equation}
According to \eqref{Bound1} in Theorem \ref{thm:heat}, the covariance bound is proportional to $1/N$ and the coefficient is provided by letting $\beta\to\infty$. Therefore, assuming the initial covariance is zero, we have the following result
\begin{equation}\label{bounds_mean_field}
  |\text{cov} (g(\bfx_i (t)), g(\bfx_j(t)))|\leq 2\left(\frac{e^{-at}-1}{-a}-\frac{e^{-(a+w)t}-1}{-a-w}\right)\sigma_u^2\frac{1}{N}
\end{equation}
In Figure \ref{Gaussian_Hov_MeanField_Hov}, we show the covariance between $u_1$ and $u_i$ for $i=1,\ldots, N$ at $t=5$. Here the initial conditions are zero everywhere. Different columns show the covariance with different $N$. The spatiotemporal simulations are also included to provide intuitions. Without the diffusion and the initial covariance, the covariance between $u_1$ and all other $u_i$ with $i\neq1$ has the same non-zero value due to the global effect from the mean field interaction.

In Figure \ref{Gaussian_Hov_MeanField_Curve}, the dependence of the covariance between $u_1$ and $u_2$ as a function of $N$ is shown. It is clear that the covariance decays as a function of $1/N$, which validates the results in Theorem \ref{thm:heat} (and that in \eqref{bounds_mean_field}). The black dashed line in panel (b) (log-log scale) shows the bound in \eqref{bounds_mean_field}, where the constant $C_1$ is given by the coefficient on the right hand side of \eqref{bounds_mean_field} in front of $1/N$.
\begin{figure}[!h]
\centering
\hspace*{-2cm}
{\includegraphics[width=20cm]{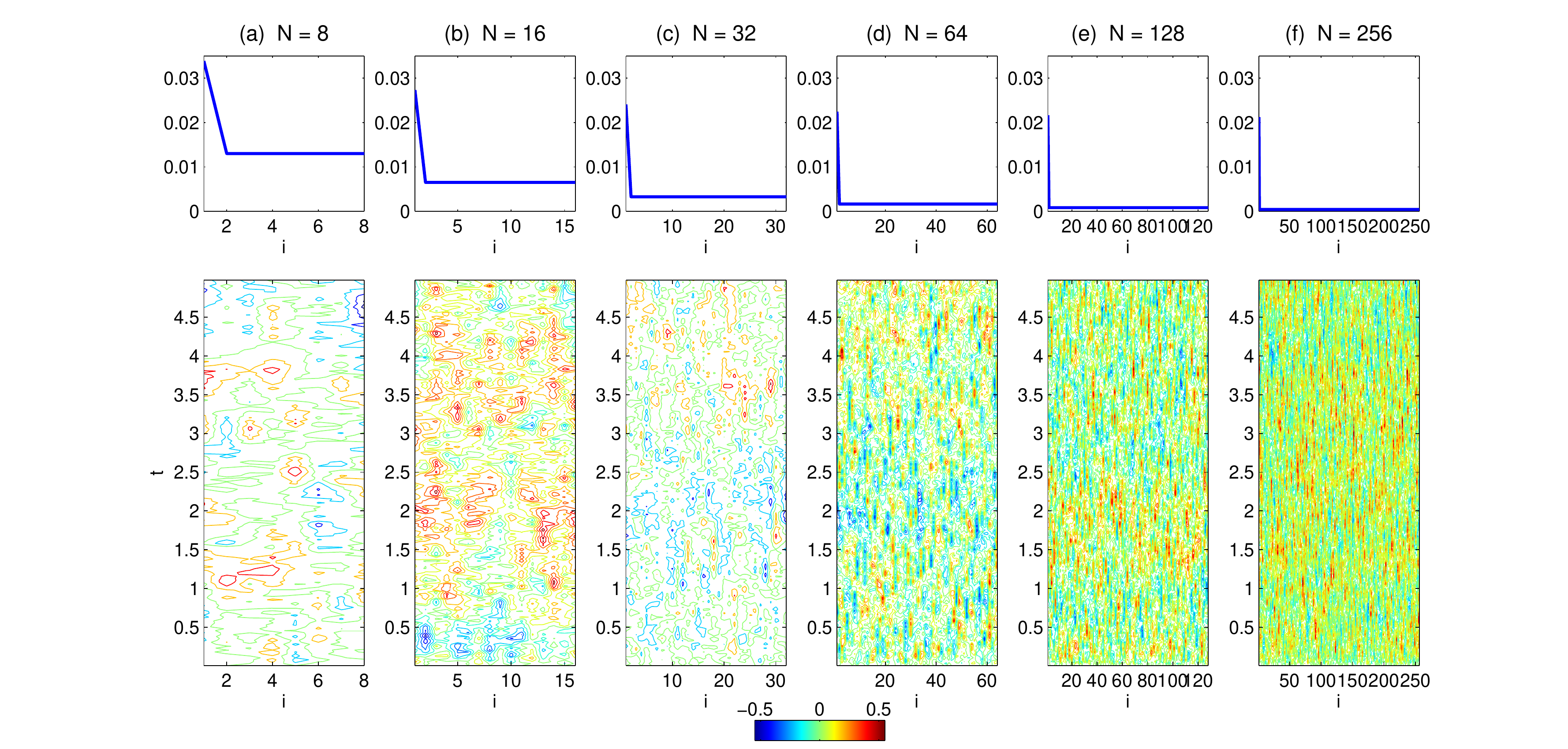}}
\caption{Linear model with only mean field interaction. The covariance between $u_1$ and $u_i$ for $i=1,\ldots, N$ at $t=5$. The spatiotemporal simulations are also included.}\label{Gaussian_Hov_MeanField_Hov}
\end{figure}

\begin{figure}[!h]
\centering
\hspace*{-2cm}
{\includegraphics[width=20cm]{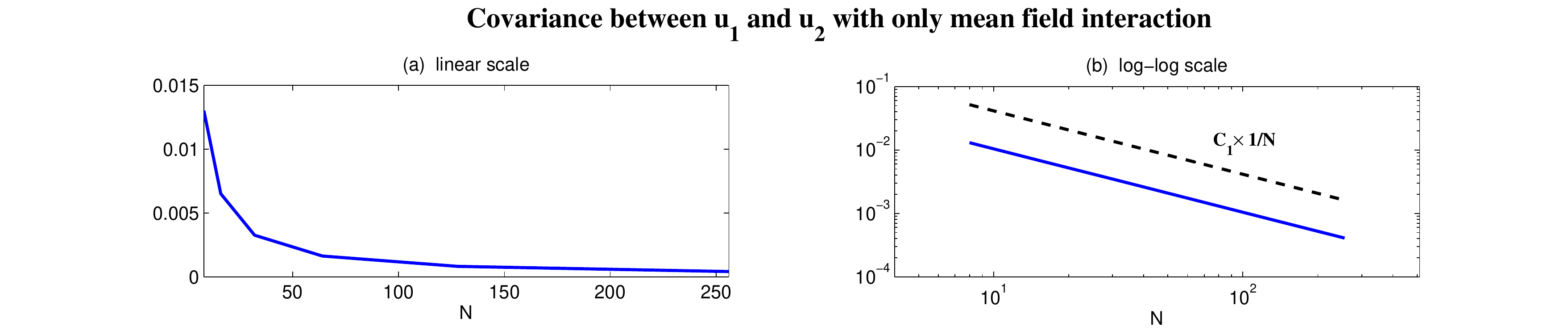}}
\caption{Linear model with only mean field interaction. The dependence of the covariance between $u_1$ and $u_2$ as a function of $N$. The black dashed line in panel (b) (log-log scale) shows the bound in \eqref{bounds_mean_field}, where the constant $C_1$ is given by the coefficient on the right hand side of \eqref{bounds_mean_field} in front of $1/N$.}\label{Gaussian_Hov_MeanField_Curve}
\end{figure}

\noindent\textbf{2. Linear model with only diffusion.} Next, we study the case with only diffusion. In other words, we set the mean field interaction coefficient $w = 0$. The constants in \eqref{linear_model_constant1} reduce to
\begin{equation}\label{linear_model_constant1_diffusion}
  \lambda_0 = -a  - 2d_u,\qquad \lambda_F = d_u,\qquad\mbox{and}\qquad \lambda_H = 0.
\end{equation}
According to \eqref{Bound2} in Theorem \ref{thm:heat}, the covariance bound is proportional to $e^{-\beta\mathbf{d}(i,j)}$. Again, assuming the initial covariance is zero, we have the following result
\begin{equation}\label{bounds_diffusion}
  |\text{cov} (g(\bfx_i (t)), g(\bfx_j(t)))|\leq \frac{2(e^{\lambda_\beta t}-1)\sigma_u^2}{\lambda_\beta}e^{-\beta\mathbf{d}(i,j)},
\end{equation}
where
\begin{equation*}
  \lambda_\beta = \lambda_0 + \lambda_F (e^\beta + e^{-\beta})
\end{equation*}
In Figure \ref{Gaussian_Hov_Diffusion_Hov}, we show the covariance between $u_1$ and $u_i$ for $i=1,\ldots, N$ at $t=5$ with $N=64$ fixed. Here the initial conditions are zero everywhere. Different columns show the covariance with different strengths of the diffusion coefficient $d_u$. The spatiotemporal simulations are also included to provide intuitions. Without the mean field interaction and the initial covariance, the covariance between $u_1$ and $u_i$ for $u_i$ being far from $u_1$ decays to zero, reflecting the local contribution due to only the diffusion (without mean field interaction). On the other hand, the increase of $d_u$ allows the increase of the covariance between $u_1$ and $u_i$ for $i$ far from $1$. This is clearly illustrated in the spatiotemporal patterns and is consistent with our intuition as well.

Figure \ref{Gaussian_Hov_Diffusion_Curve} shows the covariance between $u_1$ and $u_i$ for $i = 1,\ldots, N/2$ with $d_u=20$ and $N=64$ at $t=5$. In panel (b), the logarithm scale is shown, which indicates a linear dependence of the log covariance on the distance $\mathbf{d}(1,i)$. This validates the exponential decay of the covariance as a function of the distance as shown in \eqref{bounds_diffusion}. The black dashed line shows the bound in \eqref{bounds_diffusion}, where the constant $C_2$ is given by the coefficient on the right hand side of \eqref{bounds_diffusion} in front of $e^{-\beta\mathbf{d}(i,j)}$ by taking $\beta=1/5$. We have tested other $\beta$, with which the bounds all above the actual covariance. Therefore, we validate the theoretical results with only diffusion, described by \eqref{Bound2}.

\begin{figure}[!h]
\centering
\hspace*{-2cm}
{\includegraphics[width=20cm]{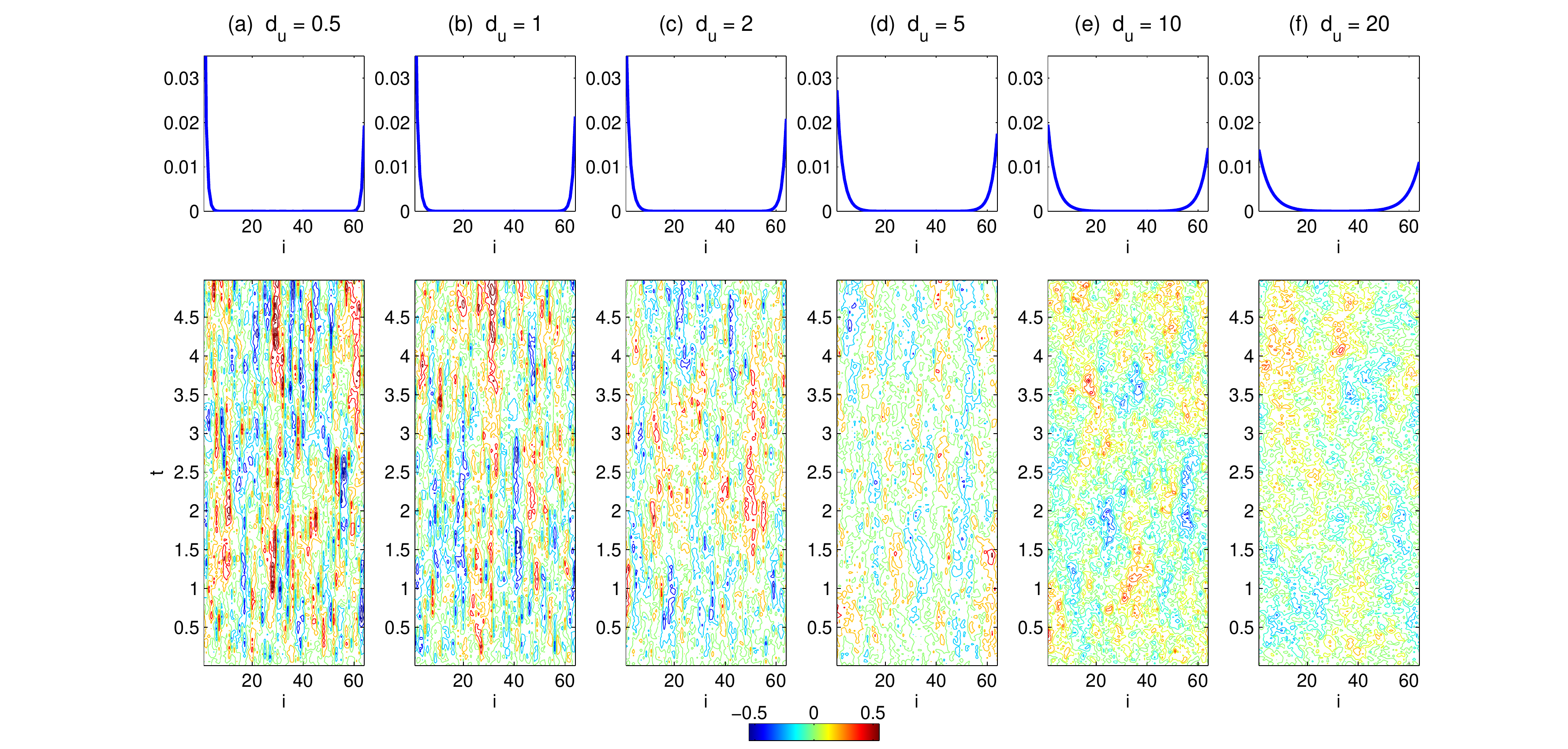}}
\caption{Linear model with only diffusion. The covariance between $u_1$ and $u_i$ for $i=1,\ldots, N$ at $t=5$ with $N=64$ fixed. }\label{Gaussian_Hov_Diffusion_Hov}
\end{figure}

\begin{figure}[!h]
\centering
\hspace*{-2cm}
{\includegraphics[width=20cm]{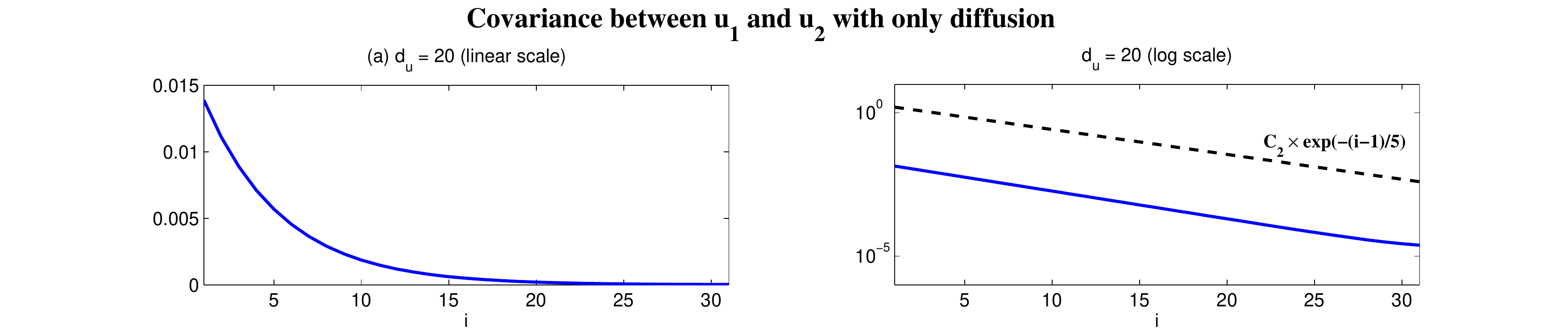}}
\caption{Linear model with only diffusion. The covariance between $u_1$ and $u_i$ for $i = 1,\ldots, N/2$ with $d_u=20$ and $N=64$ at $t=5$. The black dashed line in panel (b) (logarithm scale) shows the bound in \eqref{bounds_diffusion}, where the constant $C_2$ is given by the coefficient on the right hand side of \eqref{bounds_diffusion} in front of $e^{-\beta\mathbf{d}(i,j)}$ by taking $\beta=1/5$. }\label{Gaussian_Hov_Diffusion_Curve}
\end{figure}

\noindent\textbf{3. Linear model with both the mean field interaction and diffusion.} Finally, we study the case with both mean field interaction and diffusion. Figure \ref{Gaussian_Hov_Both_Hov} is similar to Figure \ref{Gaussian_Hov_Diffusion_Hov} but with a non-zero mean field interaction $w=5$. Therefore, unlike the results in Figure \ref{Gaussian_Hov_Both_Hov}, the covariance between $u_1$ and $u_i$ where $i$ is far from $1$ is non-zero due to the mean field interaction that has a global impact. In Figure \ref{Gaussian_Hov_Both_Curve}, it is also clear that when $i$ increases the covariance between $u_i$ and $u_1$ first experience an exponential decay and then the covariance remains as a constant. Thus, the exponential decay is the main contribution to the covariance behavior for grid points that are close to $i=1$ while the mean field interaction plays the dominant role at the location that is far away. These simulations are consistent with the theoretical results shown in Theorem \ref{thm:heat}.

\begin{figure}[!h]
\centering
\hspace*{-2cm}
{\includegraphics[width=20cm]{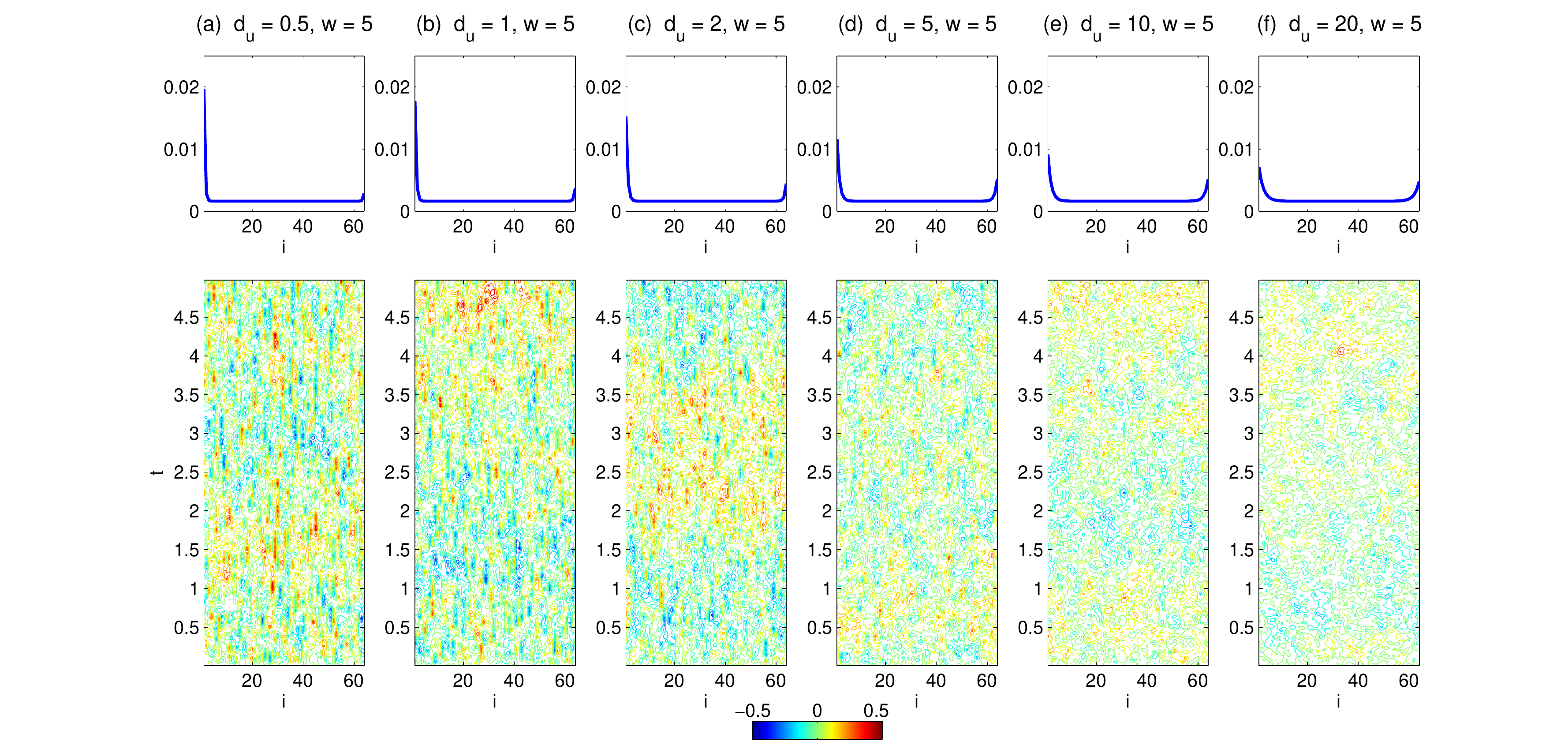}}
\caption{Linear model with both mean field interaction and diffusion. The covariance between $u_1$ and $u_i$ for $i=1,\ldots, N$ at $t=5$ with $N=64$ fixed and the mean field interaction $w=5$. }\label{Gaussian_Hov_Both_Hov}
\end{figure}

\begin{figure}[!h]
\centering
\hspace*{-2cm}{\includegraphics[width=20cm]{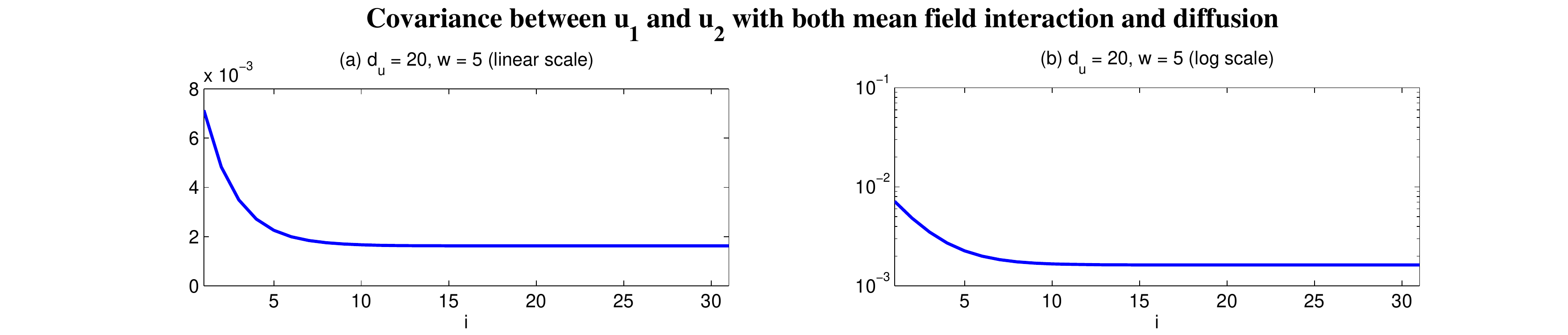}}
\caption{Linear model with both mean field interaction and diffusion. The covariance between $u_1$ and $u_i$ for $i = 1,\ldots, N/2$ with $d_u=20, w=5$ and $N=64$ at $t=5$. The local interaction is the main contributor to the covariance behavior at the beginning while the mean field interaction plays the dominant role at the location that is far away.}\label{Gaussian_Hov_Both_Curve}
\end{figure}
\clearpage

\subsection{A stochastically coupled FHN model}
In this subsection, we explore the covariance structure in a stochastically coupled FitzHugh-Nagumo (FHN) model, which is a nonlinear model that describes activation and deactivation dynamics of spiking neurons \cite{lindner2004effects, CM17pnas}.   The model reads,
\begin{equation}\label{FHN_model}
\begin{split}
  \epsilon du_i &= \left(u_i -\frac{1}{3} u_i^3 - v_i + d_u(u_{i+1}-2u_i + u_{i-1}) + w(\bar{u} - u_i) \right) dt+ \sqrt{\epsilon} \delta_1 dW_{u_i}, \\
  dv_i &= (u_i + a)dt+ \delta_2 dW_{v_i},\qquad\mbox{for~} i = 1,\ldots, N,
\end{split}
\end{equation}
where $u_i$ and $v_i$ are activator and inhibitor variables. Periodic boundary conditions are imposed on $u_i$ variables. In \eqref{FHN_model}, the timescale ratio $\epsilon = 0.01\ll 1$ leads to a slowfast structure of the model. The parameter $a=1.05>1$ such that the system has a global attractor in the absence of noise and diffusion \cite{guckenheimer2013nonlinear}. The random noise is able to drive the system above the threshold level of global stability and triggers limit cycles intermittently. Note that with $N=1$, the model reduces to the classical FHN model with a single neuron and it contains the model families with both coherence resonance and self-induced stochastic resonance \cite{deville2005two}. With different choices of the noise strength $\delta_1, \delta_2$
and the diffusion coefficient $d_u$, the system in \eqref{FHN_model} exhibits rich dynamical behaviors. Finally, in \eqref{FHN_model}, the variable $\bar{u} = \frac{1}{N}\sum_{i=1}^Nu_i$ is the mean field interaction.

The goal of this subsection is to study the spatial structure of the covariance due to both the local effect (i.e., diffusion) and the global effect (i.e., mean field interaction).

\subsubsection{Basic properties}
For the convenience of stating the theoretical results, we  slightly change of notation of \eqref{FHN_model} in this paragraph,
\begin{equation}
\label{eqn:FHNmeanfield}
\begin{gathered}
 du_i=\left(\frac1\epsilon u_i- \frac1{3\epsilon} u_i^3-\frac{1}{\sqrt{\epsilon}}v^\epsilon_i+\frac{d_u}{\epsilon}(u_{i+1}+u_{i-1}-2u_i)+\frac{\tilde{w}}{\epsilon}(\bar{u}-u_i)\right)dt+\frac1{\sqrt{\epsilon}}\delta_1 dW_{u_i},\\
dv^\epsilon_i=\left(\frac1{\sqrt{\epsilon}}u_i+ \frac a{\sqrt{\epsilon}}\right) dt+\frac{\delta_2}{\sqrt{\epsilon}} dW_{v_i}.
\end{gathered}
\end{equation}

To apply the framework we developed, it is natural to let $\bfx_i=(u_i,v_i^\epsilon)^T$ such that $q=2$, and
\begin{equation}
\label{tmp:FHNhessian}
\nabla_{\bfx_i} \bff(t,\bfx_{i-1},\bfx_i,\bfx_{i+1})=\begin{bmatrix} \epsilon^{-1}- \epsilon^{-1}u_i^2- 2 \epsilon^{-1}d_u- \epsilon^{-1}\tilde{w} &-\epsilon^{-1/2}\\
\epsilon^{-1/2} &0
   \end{bmatrix}
\end{equation}
and
\[
\nabla_{\bfx_{i\pm 1}} \bff(t,\bfx_{i-1},\bfx_i,\bfx_{i+1})=\begin{bmatrix}   \epsilon^{-1}d_u &0 \\
0 &0
   \end{bmatrix},
  \quad
  \nabla_{\bfx_i} h=\begin{bmatrix}  \tilde{w} &0 \\
0 &  0
   \end{bmatrix},
  \quad
   \Sigma=
   \begin{bmatrix}   \epsilon^{-1/2}\delta_1 &0 \\
0 &  \epsilon^{-1/2} \delta_2
   \end{bmatrix}.
\]
Since the $-\epsilon^{-1} u_i^2$ is always negative, we can take
\[
\lambda_0=\epsilon^{-1}\max\{1-2d_u-\tilde{w},0\}, \quad \lambda_F=\epsilon^{-1} d_u,\quad \lambda_H=\epsilon^{-1} \tilde{w}.
\]
Therefore, using the spatial averaging sampling strategy for high dimensional systems, we have
\begin{cor}
\label{cor:FHN}
The block average estimator is consistent for the FHN system \eqref{eqn:FHNmeanfield}. In particular, estimate \eqref{eqn:estbound} holds with
$\|\Sigma^2\|_F=\sqrt{\delta_1^4+\delta^4_2}/\epsilon $, and
\[
\eta_\beta=\epsilon^{-1}\max\{1-2d_u, \tilde{w}\}+2\epsilon^{-1} d_u(e^\beta+e^{-\beta}),\quad \lambda_\beta=\eta_\beta-\epsilon^{-1}\tilde{w}.
\]
\end{cor}
In particular, from the bound \eqref{eqn:estbound}, we learn that the estimator consistency deteriorates with larger $t$, $d_u$ and $\tilde{w}$, which correspond to longer simulation time, stronger local interaction, and stronger mean field interaction regimes.

\subsubsection{The FHN model with different diffusion and mean field interaction coefficients}
\noindent\textbf{1. FHN model with only diffusion}.
First, we focus on the situation in the FHN model where the mean field interaction is zero ($w=0$). Therefore, the covariance comes purely from the local effect, namely the diffusion. We set $\delta_1 = \delta_2 = 0.4$ but choose different values of the diffusion parameter $d_u$ which provides different dynamical behavior:
\begin{equation}\label{FHN_Diffusion}
\begin{split}
  \mbox{(a) Strongly mixed regime:}\qquad& d_u = 0.02,\\
  \mbox{(b) Weakly coherent regime:}\qquad&  d_u = 0.5,\\
  \mbox{(c) Strongly coherent regime:}\qquad& d_u = 10.
\end{split}
\end{equation}
The simulation of the model in these three dynamical regimes is shown in Figure \ref{FHN_Regimes_Diffusion}. Figure \ref{Cov_diffusion_T5} shows the covariance between $u_1$ and $u_i$ for $i=1,\ldots,N$ at $t = 5$. Here we use a Monte Carlo simulation with $M=8192$ samples at each grid points (namely, running the Monte Carlo for $M$ times). With the increase of $d_u$ from Regime (a) to Regime (c), the spatial covariance of $u$ increases. In particular, in Regime (c) with a strong diffusion, the covariance between $u_1$ and all other $u_i$ for $i=1,\ldots, N$ remains significant, and is above $0.5$. Note that there is almost no mixing of the underlying stochastic system with such a large diffusion coefficient in Regime (c) (See Figure \ref{FHN_Regimes_Diffusion}), which may cause some issues of the controllability of the system. In fact, with a large diffusion coefficient, the role of the stochastic forcing becomes relatively weaker and the control of the system via the external stochastic forcing becomes more difficult. This also leads to some potential issues in applying the spatial averaging strategy for computing the covariance, especially for long time. Note that although Corollary \ref{cor:FHN} is always valid, the constant $\eta_\beta$ becomes very large with a large diffusion coefficient (due to the second term) and thus the consistency becomes nearly degenerated very quickly. See the remarks below Corollary \ref{cor:FHN} and the statement of the bound in \eqref{eqn:estbound}. The numerical results using the spatial averaging strategy will be illustrated in the next subsection. On the other hand, the spatial covariance of $v$ changes in the same fashion as $u$ but is weaker. This is because the diffusion appears in the $u$ equations and it only has an indirect impact on $v$. In all three regimes, the spatial covariance decays in an exponential rate (see the subpanels). All these are consistent with the result in \eqref{Bound2} of Theorem \ref{thm:heat}.
\begin{figure}[!h]
\centering
\hspace*{-2cm}{\includegraphics[width=20cm]{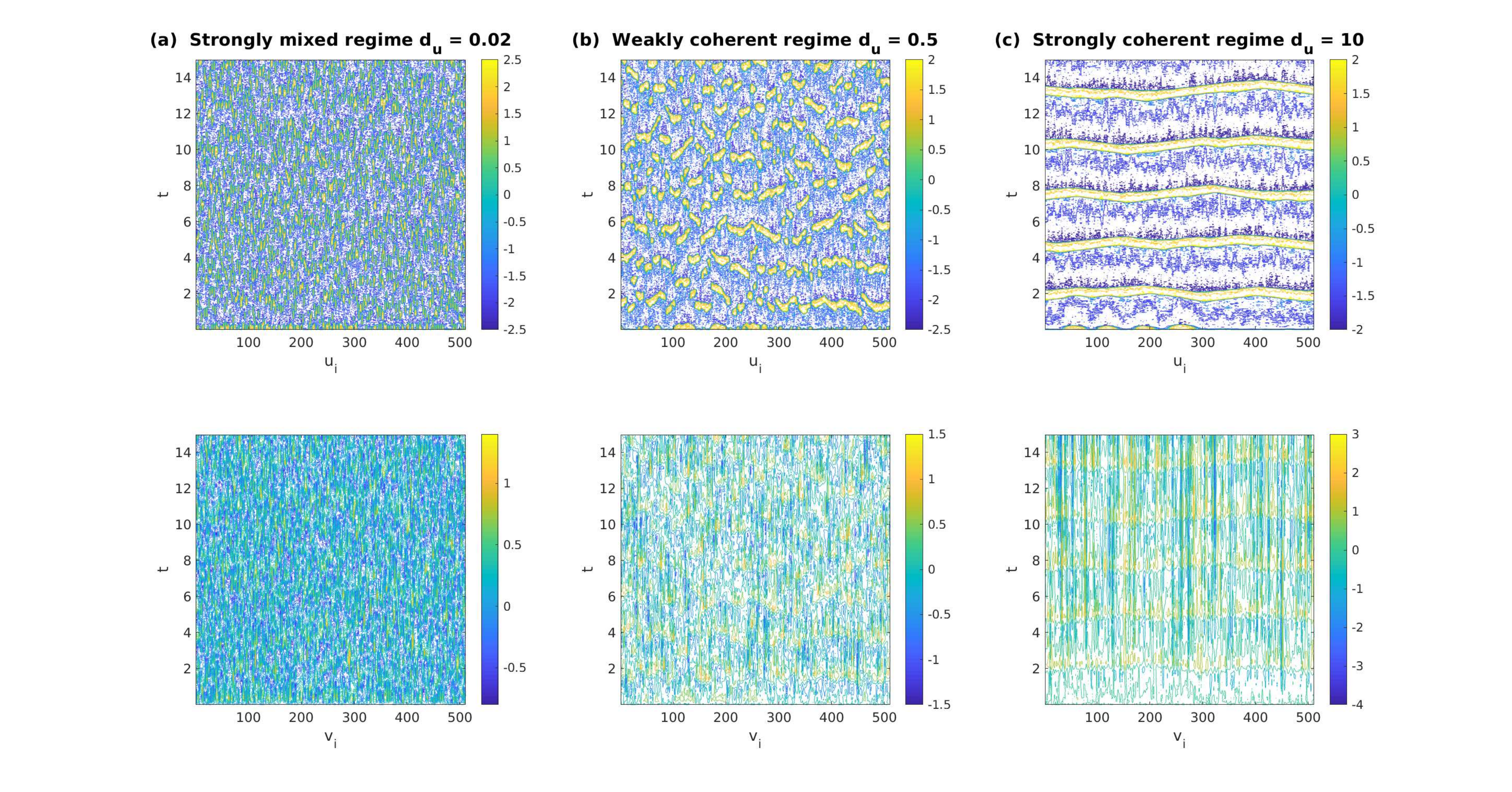}}
\caption{Simulations of the FHN model \eqref{FHN_model} with only diffusion but no mean field interaction ($w=0$) in (a) Strongly mixed regime, (b) Weakly coherent regime, and (c) Strongly coherent regime. The first row is the simulation of $u_i$ and the second is $v_i$. Here $N=512$.}\label{FHN_Regimes_Diffusion}
\end{figure}

\begin{figure}[!h]
\centering
\hspace*{-2cm}{\includegraphics[width=20cm]{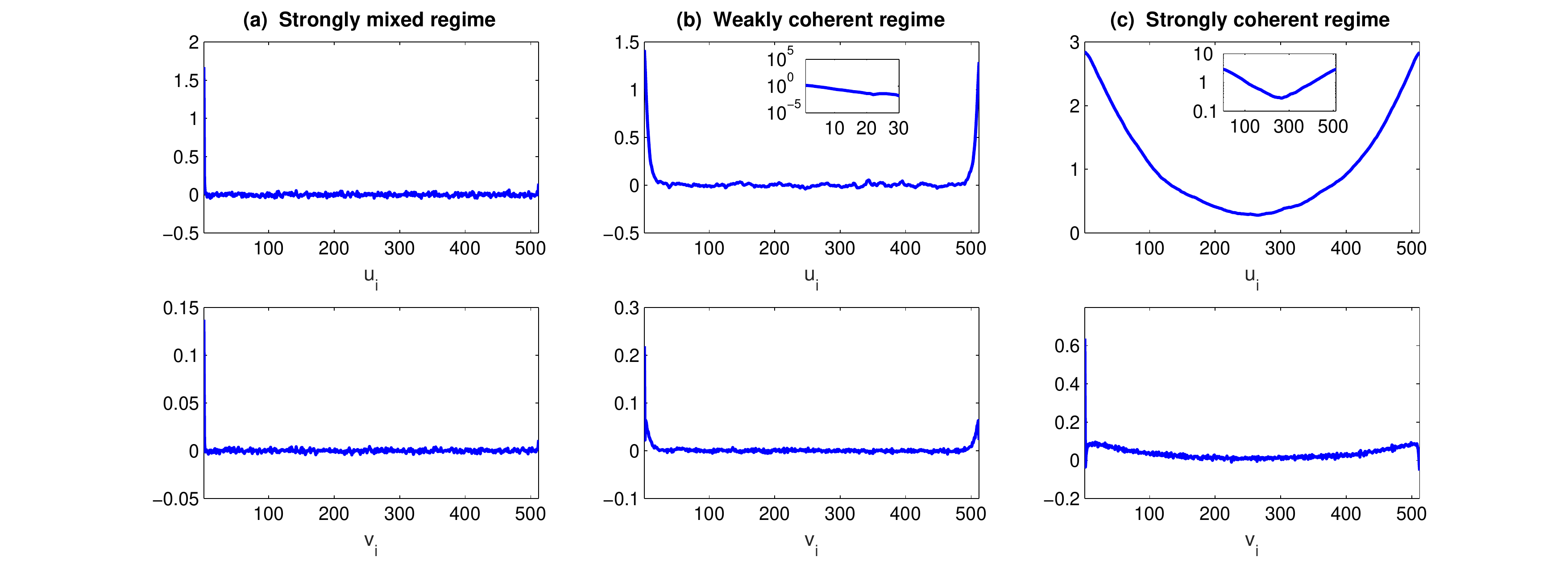}}
\caption{Covariance between $u_1$ and $u_i$ for $i=1,\ldots,N$ in the FHN model \eqref{FHN_model} with only diffusion but no mean field interaction ($w=0$) in the three regimes described in Figure \ref{FHN_Regimes_Diffusion} at time $t=5$. The first row is the simulation of $u_i$ and the second is $v_i$.  The subpanels show the covariance in a logarithm scale. }\label{Cov_diffusion_T5}
\end{figure}

\noindent\textbf{2. FHN model with only mean field interaction}.
Now we study the situation in the FHN model where the diffusion is zero ($d_u=0$). Therefore, the covariance comes purely from the global effect (i.e., the mean field interaction). We set again $\delta_1 = \delta_2 = 0.4$ but choose different values of the mean field interaction parameter $w$ which provides different dynamical behavior:
\begin{equation}\label{FHN_MeanField}
\begin{split}
  \mbox{(a) Weak mean field interaction regime:}\qquad& w = 0.1,\\
  \mbox{(b) Moderate mean field interaction regime:}\qquad&  w = 0.3,\\
  \mbox{(c) Strong mean field interaction regime:}\qquad& w = 0.5.
\end{split}
\end{equation}
The simulations of the three regimes are shown between Figure \ref{FHN_Regimes_MeanField}. Again, these results are computed using a Monte Carlo  simulation with $M=8192$ samples at each grid points. In the moderate mean field interaction regime (Regime (b)), some global effects can already been seen. In the strong mean field interaction regime (Regime (c)), it is clear that the variables at different grid points are highly synchronized with each other, which is consistent with the theoretic analysis described in the previous section.

Note the difference in Figure \ref{FHN_Regimes_Diffusion} and \ref{FHN_Regimes_MeanField}. In Figure \ref{FHN_Regimes_Diffusion} the diffusion leads to the development of local interaction and extends it to the global scale. In Figure \ref{FHN_Regimes_MeanField} the mean field interaction starts directly with a change in the global scale and such a change becomes significant with the increase of the mean field interaction strength $w$. Again, the constant $\eta_\beta$ in Corollary \ref{cor:FHN} with a large mean field interaction coefficient can become very large (due to the first term) and thus the consistency becomes nearly degenerated.

%In Figures \ref{Plot_d_m_03_T5}, we show the covariance between $u_1$ and $u_i$ and that between $v_1$ and $v_i$ for $i=1,\ldots, N$ with $w = 0.3$. Again, we use a Monte Carlo simulation with $M=8192$ samples at each grid points (namely, running the Monte Carlo for $M$ times). We have the following results. First, the mean field interaction leads to a global non-zero covariance for both $u$ and $v$. In addition, the covariance between $u_1$ and all different $u_i, i=2,\ldots,N$ are essentially the same.

Figure \ref{Cov_meanfield_T3T5} shows the covariance $\mbox{Cov}(u_1,u_2)$ and $\mbox{Cov}(v_1,v_2)$ as a function of $N$ in two different regimes with either moderate or strong mean field interactions.
With the increase of $N$, the covariance decreases and the decrease rate is proportion to $1/N$ (see the log-log plot), which is consistent with the theoretical conclusion in \eqref{Bound1} of Theorem \ref{thm:heat}.

%Note: When we compute $\mbox{Cov}(u_1,u_2)$, we actually compute $\mbox{Cov}(u_1,u_i)$ for $2=1,\ldots,N$ and then take the average. This allows us to average out the numerical error since we know in the absence of diffusion the covariance $\mbox{Cov}(u_1,u_i)$ for different $i$ is the same. The same trick applies to compute  $\mbox{Cov}(v_1,v_2)$.

\begin{figure}[!h]
\centering
\hspace*{-2cm}{\includegraphics[width=20cm]{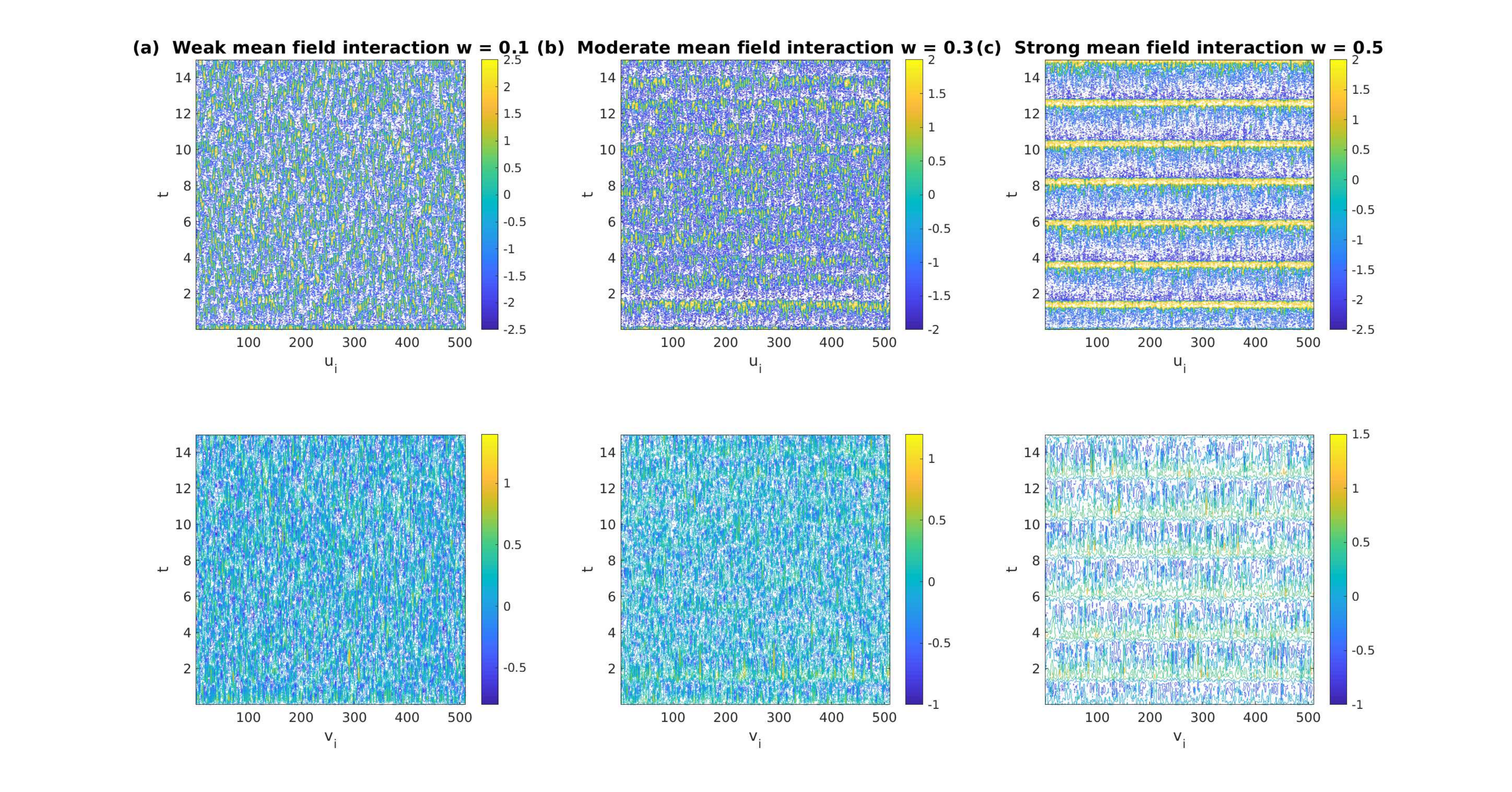}}
\caption{Simulations of the FHN model \eqref{FHN_model} with only mean field interaction but no diffusion $d_u=0$ in (a) Weak mean field interaction regime, (b) Moderate mean field interaction regime, and (c) Strong mean field interaction regime. The first row is the simulation of $u_i$ and the second is $v_i$. Here $N=512$.}\label{FHN_Regimes_MeanField}
\end{figure}

\begin{figure}[!h]
\centering
\hspace*{-2cm}{\includegraphics[width=20cm]{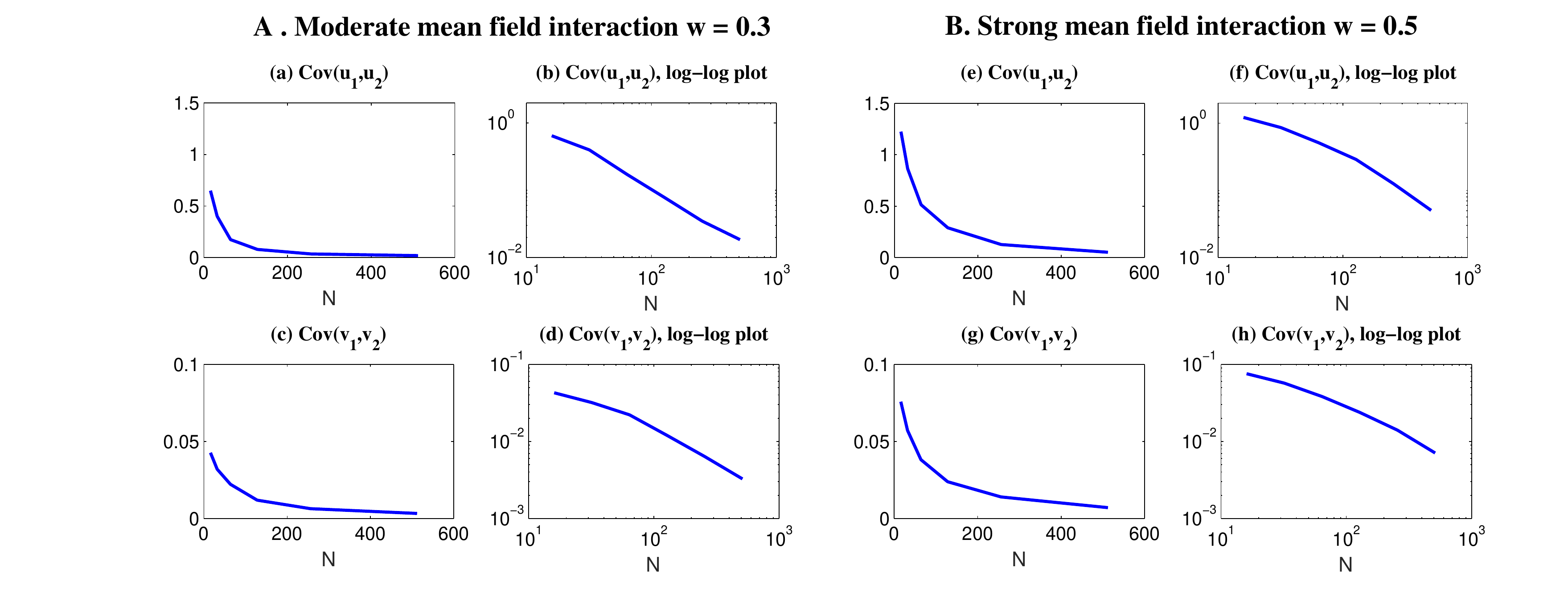}}
\caption{Covariance between $u_1$ and $u_2$ (panel (a)) and between $u_1$ and $u_2$ (panel (c)) as a function of $N$ in the regime with moderate mean field interaction $w=0.3$ and $d_u=0$. Panels (b) and (d) show the same results but are plotted in the log-log scale in order to illustrate the decay rate $1/N$ (in log-log scale, this corresponds to a linear decay rate). Panels (e)--(h) are similar but are for the regime with strong mean field interaction $w = 0.5$ and $d_u = 0$.}\label{Cov_meanfield_T3T5}
\end{figure}

\subsubsection{Recovery of the covariance using the spatial averaging strategy}

In this subsection, we compare the covariance computed from the spatial averaging strategy with the truth  in order to study the consistency of the spatial averaging strategy. Here the truth is formed by using direct Monte Carlo simulations.

Recall in Section \ref{Sec:SpatialAveraging}, the spatial averaging strategy makes use of the statistical symmetry of the system. Comparing with the covariance computed by using the direct Monte Carlo simulations with a large sample size, the spatial averaging strategy requires only a small number of repeated experiments to reach the same accuracy (providing that the consistency is valid). In fact, due to the spatial homogeneity, the covariance between $u_1$ and $u_2$ is the same as that between $u_i$ and $u_{i+1}$ for all $i=1,\ldots, N$. Therefore, regarding $u_1, u_2, \ldots, u_N$ as a sample set, the covariance between $u_1$ and $u_2$ can be computed by taking the covariance between $u_1, u_2, \ldots, u_N$ and its shifted sample set $u_2, u_3, \ldots, u_N, u_1$. Note that if $M$ repeat experiments are combined (namely, $M$ Monte Carlo simulations) with the spatial averaging strategy, then the effective sample size is $M_{eff} = MN$. One potential issue of the spatial averaging strategy is that the samples $u_1, u_2, \ldots, u_N$ are not independent with each other. Therefore, when these samples are strongly correlated, then the spatial averaging strategy in computing the covariance may lead to biases. In this subsection, we compare the covariance computed from the spatial averaging strategy with the truth that is generated from the Monte Carlo simulations in different dynamical regimes. The goal is to explore the consistency of this spatial averaging strategy in computing the covariance in different dynamical regimes.

Below, we always take $N=512$ spatial grid points. For a fair comparison, we generate the truth using the direct Monte Carlo simulation with $M_{MC} = N = 512$ samples, which equals to the effective sample number using the spatial averaging strategy with simply one simulation.
In all the tests shown below, we compare the covariance computed using the above two strategies at time $t=0.5, t=1, t=2$ and $t=5$. The regimes to be tested contain both local and non-local effects with various strengths. See Figures \ref{FHN_Regimes_1}--\ref{FHN_Regimes_2} for the spatiotemporal simulations of different dynamical regimes.

In Figure \ref{FHN_Regimes_1}, strong spatial correlations are seen in all three dynamical regimes and the underlying stochastic processes is not mixing. Yet, there are some fundamental difference in these dynamical regimes. In Regime (a), there is a strong diffusion  but there is no mean field interaction. Thus, the correlation comes from only the local effect but the diffusion is so strong that each grid point affects the other points that are far from itself. In Regime (b), only the mean field interaction plays an role since the diffusion is zero but the mean field interaction coefficient is quite large. Thus, a strong synchronization is seen in the spatiotemporal patterns. Regime (c) involves both a moderate diffusion and a moderate mean field interaction. The combined local and non-local effects also contribute to the strong correlation in the spatiotemporal patterns. Figure \ref{FHN_Shortterm_Failcase} shows the calculated covariance based on the spatial averaging strategy (red) and the truth generated from the direct Monte Carlo simulation (green). For a short time up to $t=1$, the consistency is found in Regimes (a) and (b). But the covariance computed from the spatial averaging strategy becomes significantly different from the truth after $t=2$, which indicates the loss of the long-term consistency using the spatial averaging strategy. These results are related to the loss of practical controllability as was discussed in the previous subsection. In fact, the constant $\eta_\beta$ in Corollary \ref{cor:FHN} becomes very large in both regimes due to the strong diffusion and/or the strong mean field interaction and thus the consistency of the spatial averaging strategy becomes nearly degenerated very quickly, which has been pointed out in the remarks below Corollary \ref{cor:FHN}.
For Regime (b), due to the strong global effect (via mean field interaction), the consistency is again lost very quickly at $t=1$. Note in all these three regimes the covariance is not localized after a certain time due to either strong diffusion or strong mean field interaction (or both). Thus, the spatial averaging strategy does not work well for long time.

On the other hand, no strong spatial synchronization is observed in the three dynamical regimes shown in Figure \ref{FHN_Regimes_2}. Some weak spatial synchronization is observed in Regimes (d) and (e), where the former involves a moderate diffusion and a weak mean field interaction while the latter is given by a moderate mean field interaction but no diffusion. Regarding the spatial covariance, Figure \ref{FHN_Shortterm_Moderatecase} shows that the spatial averaging strategy is able to provide accurate solutions up to $t=2$. At $t=5$, a slight inconsistency in the solution computed from the spatial averaging strategy is found, which is  due to the degeneration of the consistency for long time as indicated in  Corollary \ref{cor:FHN}.  Thus, the spatial averaging strategy is skillful at short and moderate times but the long-term consistency cannot be guaranteed in these regimes.

Finally, in Regime (f), which involves only a moderate diffusion but no mean field interaction, the consistency of the solution using the spatial averaging strategy lasts longer and is at least up to $t=5$. See Figure \ref{FHN_Shortterm_Successcase}.

To summarize, both a strong global effect (mean field interaction) and a strong local effect (diffusion) will lead to the loss of the long-term consistency in the solutions using  the spatial averaging strategy, which has been clearly indicated in  Corollary \ref{cor:FHN} and the remarks below it. The spatial averaging strategy is skillful and can be applied for a relatively long time when the FHN model has a weak global or a moderate local effect.

\begin{figure}[!h]
\centering
\hspace*{-2cm}{\includegraphics[width=20cm]{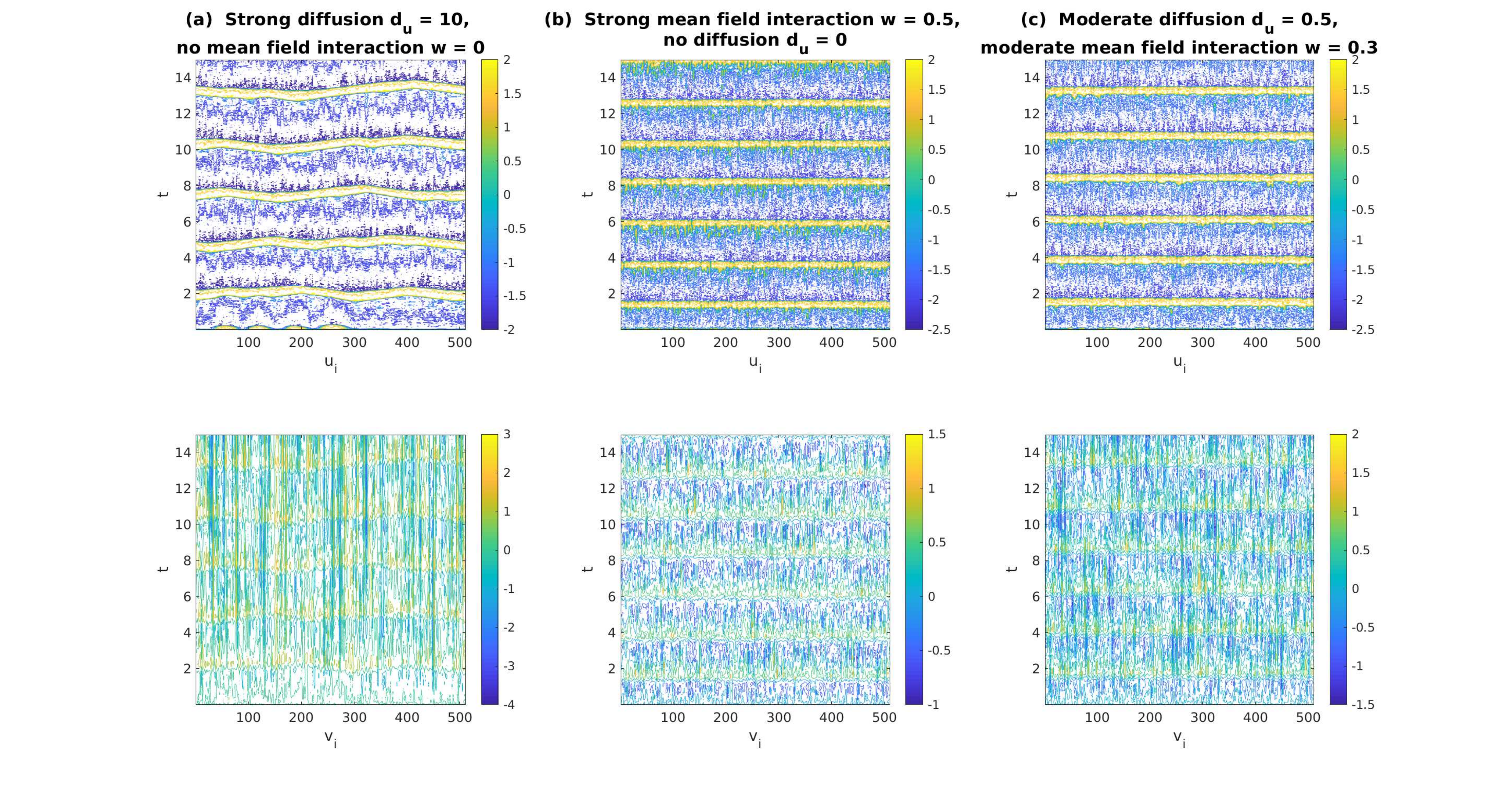}}
\caption{Spatiotemporal patterns of different dynamical regimes: (a) Strong diffusion $d_u = 10$ and no mean field interaction $w = 0$; (b) Strong mean field interaction $w = 0.5$ and no diffusion $d_u = 0$; and (c) Moderate diffusion $d_u = 0.5$ and moderate mean field interaction $w = 0.3$.}\label{FHN_Regimes_1}
\end{figure}

\begin{figure}[!h]
\centering
\hspace*{-2cm}{\includegraphics[width=20cm]{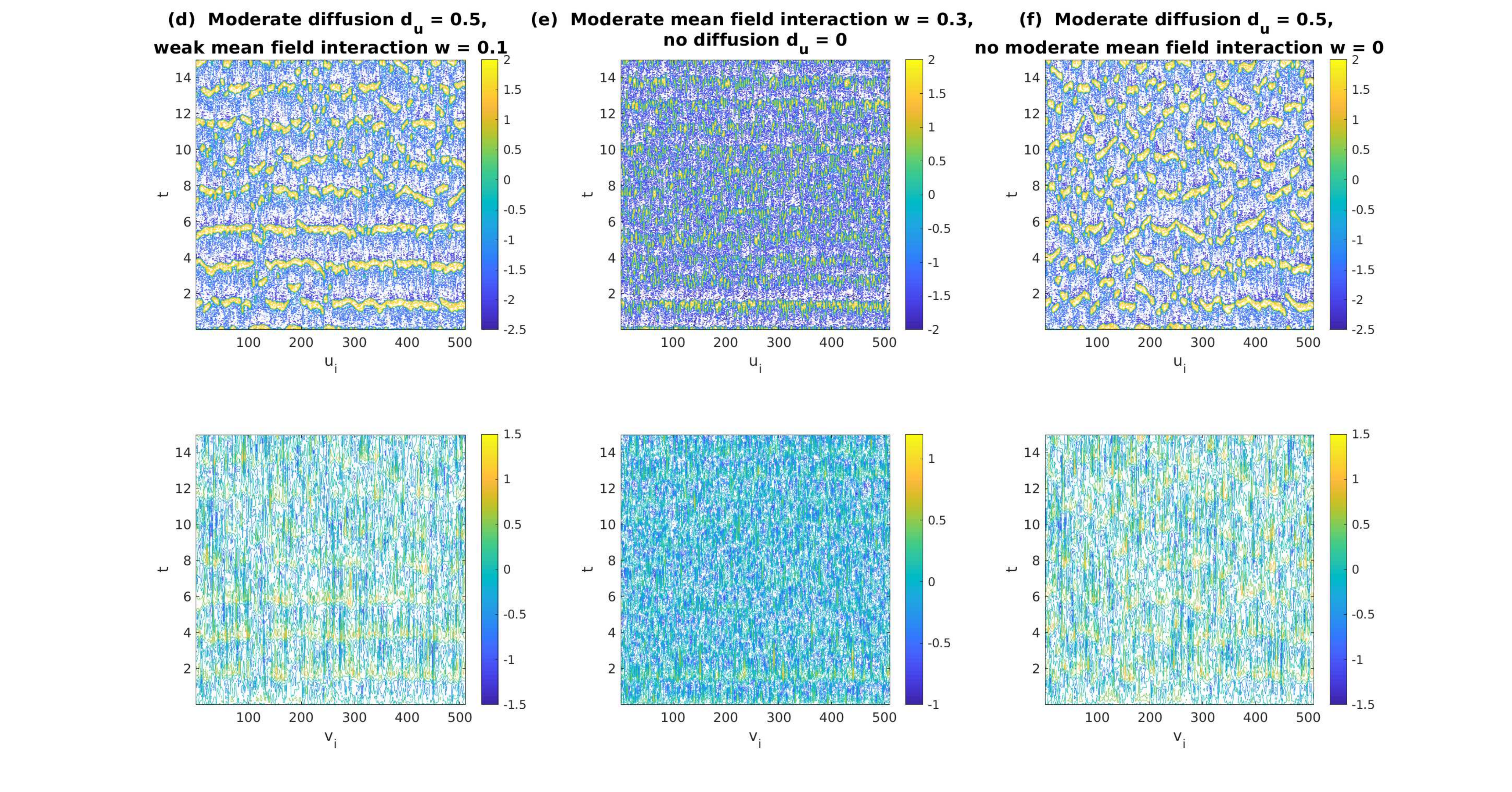}}
\caption{Spatiotemporal patterns of different dynamical regimes (continued): (d) Moderate diffusion $d_u = 0.5$ and weak mean field interaction $w = 0.1$; (e) Moderate mean field interaction $w = 0.3$ and no diffusion $d_u = 0$; and (f) Moderate diffusion $d_u = 0.5$ and no mean field interaction $w = 0$.}\label{FHN_Regimes_2}
\end{figure}

\begin{figure}[!h]
\centering
\hspace*{-2cm}{\includegraphics[width=20cm]{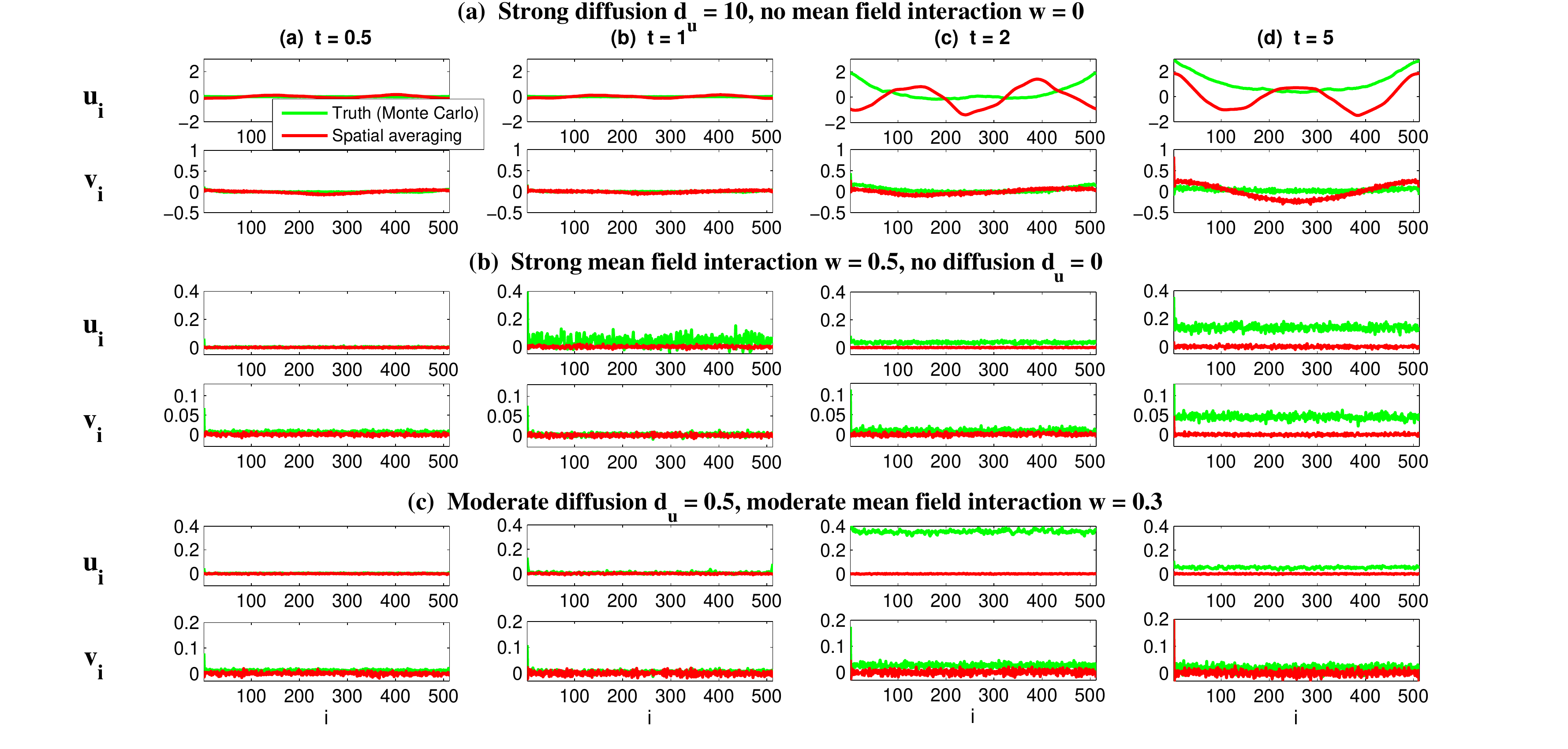}}
\caption{Comparison of the Monte Carlo and spatial average strategies in computing the covariance of the FHN model in the regimes with (a) Strong diffusion $d_u = 10$ and no mean field interaction $w = 0$; (b) Strong mean field interaction $w = 0.5$ and no diffusion $d_u = 0$; (c) Moderate diffusion $d_u = 0.5$ and moderate mean field interaction $w = 0.3$.}\label{FHN_Shortterm_Failcase}
\end{figure}

\begin{figure}[!h]
\centering
\hspace*{-2cm}{\includegraphics[width=20cm]{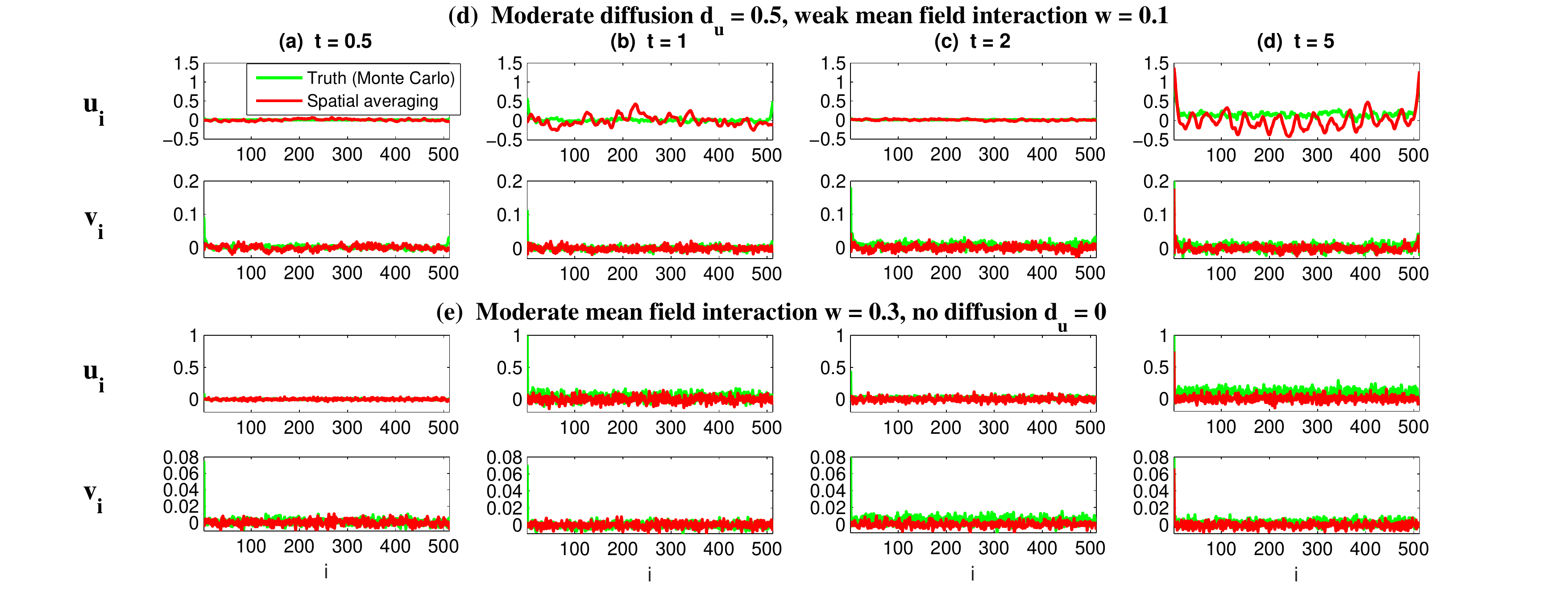}}
\caption{Comparison of the Monte Carlo and spatial average strategies in computing the covariance of the FHN model in the regimes with (d) Moderate diffusion $d_u = 0.5$ and weak mean field interaction $w = 0.1$; and (e) Moderate mean field interaction $w = 0.3$ and no diffusion $d_u = 0$; }\label{FHN_Shortterm_Moderatecase}
\end{figure}

\begin{figure}[!h]
\centering
\hspace*{-2cm}{\includegraphics[width=20cm]{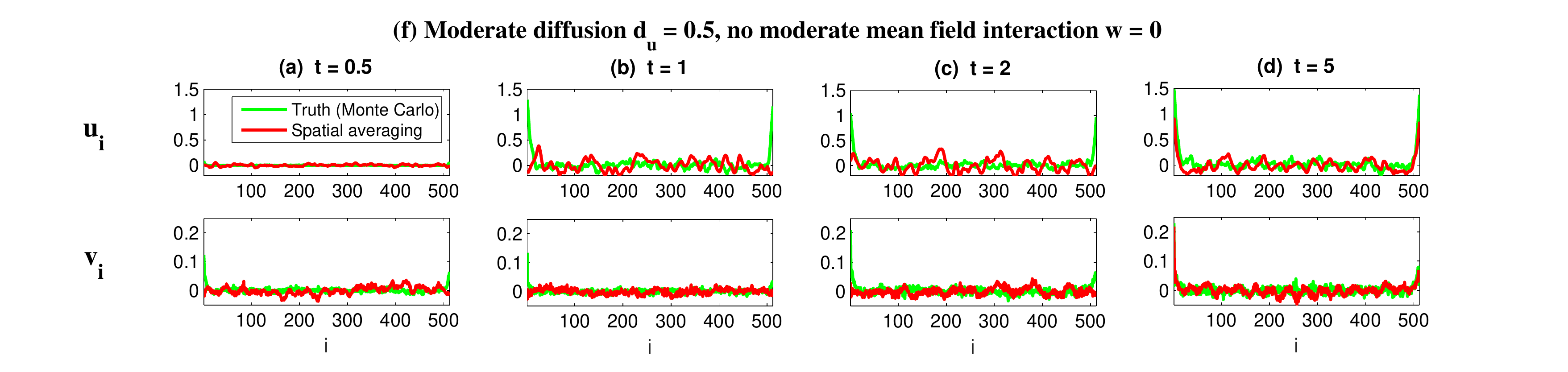}}
\caption{Comparison of the Monte Carlo and spatial average strategies in computing the covariance of the FHN model in the regimes with (f) Moderate diffusion $d_u = 0.5$ and no mean field interaction $w = 0$.}\label{FHN_Shortterm_Successcase}
\end{figure}

\section{Conclusions}\label{Sec:Conclusions}

In this article, a theoretical framework to understand the spatial localization  for a large class of stochastically coupled nonlinear systems in high dimensions is developed and is applied to nonlinear complex models for excitable media. Rigorous mathematical theories (Theorem \ref{thm:heat}) show covariance decay behavior due to both local and nonlocal effects, which result from the diffusion and the mean field interaction, respectively. The analysis is based on a comparison with an appropriate linear surrogate model, of which the covariance propagation can be computed explicitly. Two important applications of these theoretical results are illustrated. They are the spatial averaging strategy for efficiently sampling the state variables and the localization technique in data assimilation. Test examples of a  linear model and a stochastically coupled FitzHugh-Nagumo model for excitable media are adopted to validate the theoretical results. The later is also used for a systematical study of the spatial averaging strategy in efficiently sampling the covariance matrix in different dynamical regimes. The spatial averaging strategy is always skillful for short time. The long-term consistency of the spatial averaging strategy requires that the mean field interaction and the diffusion cannot be too strong and the covariance matrix has localized structures.

\section*{Acknowledgement}
The research of N.C. is supported by the Office of Vice Chancellor for Research and Graduate Education (VCRGE) at University of Wisconsin-Madison.
The research of A.J.M. is partially supported by the Office of Naval Research Grant ONR MURI N00014-16-1-2161 and the Center for Prototype Climate Modeling (CPCM) at New York University Abu Dhabi Research Institute. X.T.T is supported by NUS grant R-146-000-226-133.

\clearpage

\appendix
\section{Proof}
We provide the proofs of the theoretical results in this appendix. In particular for Theorem \ref{thm:main}, its arguments are broken into subsections

\subsection{Consistency of the numerical formulation}
\label{sub:consistent1}

\begin{lem}
\label{lem:htozero}
Under the conditions of Theorem \ref{thm:heat},  let $\bfX(0)=\bfx(0)$, and $\bfW_n=\bfw_{hn}-\bfw_{h(n-1)}$, and consider the corresponding $X(n)$ given by \eqref{eqn:update}. Then if we write $\Delta(t)=\bfX(T)-\bfx(t)$,  there is a constant $D_t$ such that for any sufficiently small $h$,
\[
\E \|\Delta(s)\|^2\leq \exp(2(\lambda_0+2\lambda_H)t) h\|D_\Sigma\|^2_F,\quad s\leq t.
\]
\end{lem}
\begin{proof}
Let $\bfz(nh)=\bfX(n)$, and $\bfz((n+s)h)=\Psi_{sh} (\bfX(n))$ for $0\leq s<1$. Clearly $\bfz(0)=\bfx(0)$. Note that
\[
d\Delta(t)=(F(\bfx(t))-F(\bfz(t)))dt+D_\Sigma d\bfw(t),\quad t\in (nh, (n+1)h).
\]
So
\[
d \|\Delta(t)\|^2=2\langle F(\bfx(t))-F(\bfz(t)), \bfx(t)-\bfz(t)\rangle dt+\|D_\Sigma\|^2_F dt+
2\langle \bfx(t)-\bfz(t), D_\Sigma d\bfw(t)\rangle.
\]
Note that under our conditions
\begin{align*}
\langle F(\bfx(t))-F(\bfz(t)), \bfx(t)-\bfz(t)\rangle=  \int^1_0 ds \langle \nabla F (\bfz(t)+s\Delta(t)) \Delta(t), \Delta(t)\rangle.
\end{align*}
But
\begin{align*}
&\langle \nabla F (\bfz(t)+s\Delta(t)) \Delta(t), \Delta(t)\rangle\\
&=\sum_{i} \Delta_i (t)^T\{ \nabla F (\bfz(t)+s\Delta(t))\}_{i,i} \Delta_i (t)+ \sum_{i,m=\pm 1} \Delta_{i+m} (t)^T\{ \nabla F (\bfz(t)+s\Delta(t))\}_{i,i+m} \Delta_i (t)\\
&\leq \sum_{i} \lambda_0 \|\Delta_i (t)\|^2+ \sum_{i,m=\pm 1} \lambda_H \|\Delta_{i+m} (t)\| \|\Delta_i (t) \|
\leq (\lambda_0+2\lambda_H) \|\Delta(t)\|^2.
\end{align*}

Therefore
\[
\langle F(\bfx(t))-F(\bfz(t)), \bfx(t)-\bfz(t)\rangle\leq (\lambda_0+2\lambda_H) \|\Delta(t)\|^2.
\]
So by Gronwall's inequality
\[
\E \|\Delta(nh+h-)\|^2\leq \exp(2(\lambda_0+2\lambda_H)h)\E \|\Delta(nh)\|^2+
\frac{ \exp(2(\lambda_0+2\lambda_H)h)-1}{2\lambda_0+4\lambda_H }\|D_\Sigma\|^2_F.
\]
Next notice that
\[
\Delta((n+1)h)=\Delta(nh)+\int^h_0 \left(F(\bfx(nh+s))-F(\bfz(nh+s)) \right)ds=\Delta((n+1)h-)-
 \int^h_0 D_\Sigma d\bfw_{nh+s}.
\]
By Ito's isometry
\[
d\left\langle \Delta(nh+s), \int^s_0 D_\Sigma d\bfw_{nh+s}\right\rangle
=\left\langle F(\bfx(nh+s))-F(\bfz(nh+s)), D_\Sigma d\bfw_{nh+s}\right\rangle ds
+\|D_\Sigma\|^2_F ds.
\]
Therefore
\[
\E \left\langle \Delta(nh+s), \int^s_0 D_\Sigma d\bfw_{nh+s}\right\rangle=
\|D_\Sigma\|^2_F s.
\]
Therefore, when $h$ is small enough, by Taylor expansion,
\begin{align*}
\E \|\Delta((n+1)h)\|^2
&=\E \|\Delta((n+1)h-)\|^2-\|D_\Sigma\|^2_F h\\
&\leq \exp(2(\lambda_0+2\lambda_H)h)\E \|\Delta(nh)\|^2
+\left(\frac{ \exp(2(\lambda_0+2\lambda_H)h)-1}{2\lambda_0+4\lambda_H }-h\right)\|D_\Sigma\|^2_F\\
&\leq \exp(2(\lambda_0+2\lambda_H)h)\E \|\Delta(nh)\|^2+4(\lambda_0+2\lambda_H)h^2\|D_\Sigma\|^2_F.
\end{align*}
Then by the discrete Gronwall's inequality, if $\lambda_0+2\lambda_H>0$, at $Th=t$, since $\exp(a)-1>a$ for all $a> 0$,
\[
\E \|\Delta(t)\|^2\leq 2(\lambda_0+2\lambda_H)\frac{\exp(2(\lambda_0+2\lambda_H)t)-1}{\exp(2(\lambda_0+2\lambda_H)h)-1}h^2\|D_\Sigma\|^2_F\leq \exp(2(\lambda_0+2\lambda_H)t)h \|D_\Sigma\|^2_F.
\]

\end{proof}

\subsection{Decomposition of covariance}
Clearly, we can view $\bfX(T)$ as a function of the realization of $\bfW_i$,
\begin{align*}
\bfX(T)=\Phi (\bfW_0,\bfW_1,\bfW_2,\cdots, \bfW_T)=\Psi_h(\cdots(\Psi_h(\Psi_h(\bfX_0)+\Sigma_h\bfW_1)+\Sigma_h\bfW_2)\cdots)+\Sigma_h\bfW_T.
\end{align*}
For the simplicity of notation, we write $\bfW_0,\bfW_1,\bfW_2,\cdots, \bfW_T$ as $\bfW_{0:T}$ in below. And the $i$-th block of $\Phi\in \reals^{qN}$ as $\Phi_i\in \reals^q$. Then we can write $\bfX_i(T)=\Phi_i(\bfW_{0:T})$.

We can study the dependence of $\bfX(T)$ on the $\bfW_i$ through the gradients, written as $\nabla_{i}\Phi:=\nabla_{\bfW_i}\Phi$. By the chain rule, it has this back propagation formulation,
\begin{equation}
\label{eqn:gradientphi}
\begin{gathered}
\nabla_T\Phi =\Sigma_h,\quad \nabla_{T-1}\Phi =\nabla \Psi_h(\bfX_{T-1})\Sigma_h,\\
\nabla_{T-2}\Phi =\nabla \Psi_h(\bfX_{T-1})\nabla \Psi_h(\bfX_{T-2})\Sigma_h,\\
\nabla_{k}\Phi= \nabla \Psi_h(\bfX_{T-1})\cdots\nabla \Psi_h(\bfX_{k})\Sigma_h,\\
\nabla_{0}\Phi= \nabla \Psi_h(\bfX_{T-1})\cdots\nabla \Psi_h(\bfX_{0})\Sigma_0. \\
\end{gathered}
\end{equation}
To investigate the covariance,  we consider another independent sample
\[
\bfX'(T)=\Phi(\bfW'_{0:T})=\Phi (\bfW'_0,\bfW'_1,\bfW'_2,\cdots, \bfW'_T).
\]
Then clearly the covariance can be written as
\[
\text{cov}(g(\bfX_i(T)), g(\bfX_j(T)))=\E g(\bfX_i(T)) g(\bfX_j(T))-\E g(\bfX_i(T)) g(\bfX'_j(T)).
\]
We can decompose it by considering
\begin{align}
\notag
&\Delta_k:=g(\Phi_i (\bfW_{0:T}))g(\Phi_j (\bfW_{0:k},\bfW'_{k+1:T}))-g(\Phi_i (\bfW_{0:T}))g(\Phi_j (\bfW_{0:k-1},\bfW'_{k:T}))\\
\label{tmp:Deltak}
&=g(\Phi_i (\bfW_{0:T}))g(\Phi_j (\bfW_{0:k-1},\bfW_k, \bfW'_{k+1:T}))-g(\Phi_i (\bfW_{0:T}))g(\Phi_j (\bfW_{0:k-1},\bfW'_k,\bfW'_{k+1:T})).
\end{align}
Because $\bfX(T)=\Phi(\bfW_{0:T}), \bfX'(T)=\Phi(\bfW'_{0:T})$, so
\[
g(\bfX_i(T)) g(\bfX_j(T))-g(\bfX_i(T))g(\bfX'_j(T))=\sum_{k=0}^{T} \Delta_k.
\]
We will bound $\E \Delta_k$ using a Stein's lemma type of formula, Lemma \ref{lem:Gaussint}. Its proof is given below.

\subsection{Stein's lemma}

%\begin{lem}
%\label{lem:Gaussint}
%Suppose $X,Y$ are independent standard Gaussian random variables of dimension $d$, $f$ and $g$ are two smooth functions. Let  $X^\theta=\cos\theta\, X +\sin\theta\, Y$, then
%\[
%\text{cov}(f(X),g(X))=\E f(X)g(X)-f(X)g(Y)=\int^{\frac\pi2}_0 \sin \theta  \E \langle \nabla f(X), \nabla g(X^\theta)\rangle d\theta.
%\]
%\end{lem}
\begin{proof}[Proof of Lemma \ref{lem:Gaussint}]
The first part comes simply by
\[
\cov(f(X),g(X))=\E f(X)g(X)-\E f(X)\E g(Y)=\E f(X)g(X)-\E f(X) g(Y).
\]
The note that
\begin{align*}
f(X)g(Y)-f(X)g(X)=f(X)g(X^\theta)&\bigg|^{\theta=\frac\pi2}_{\theta=0}=\int^{\frac\pi2}_{0}  f(X)\frac{d}{d\theta} g(X^\theta) d\theta\\
&=\int^{\frac\pi2}_0  f(X) \langle\nabla g(X^\theta), -\sin\theta\, X+\cos\theta\, Y \rangle d\theta.
\end{align*}
Note that when $X\sim \mathcal{N}(0, \mathbf{I}_d)$, and smooth vector field $h:\reals^d\to \reals^d$, let $c=(2\pi)^{-\frac{d}{2}}$, by the divergence theorem of calculus, $\int \nabla\cdot (h(\bfx) e^{-\frac12|\bfx|^2})d\bfx=0$, therefore
\begin{equation}
\E \langle h(X), X\rangle=c\int \langle h(\bfx), \bfx\rangle e^{-\frac{1}{2}|\bfx|^2}d\bfx=-c \int \langle h(\bfx), \nabla e^{-\frac{1}{2}|\bfx|^2}\rangle d\bfx
=c\int \nabla\cdot h(\bfx) e^{-\frac12|\bfx|^2}d\bfx=\E \nabla\cdot h(X).
\end{equation}
Here, $\nabla\cdot$ denotes the divergence operator. Therefore, if we condition on the value of $Y$ and use law of total expectation
\begin{align*}
\E \langle f(X)\nabla g(X^\theta), X\rangle&=\E \E_Y \langle f(X)\nabla g(\cos\theta\, X+\sin\theta\, Y), X\rangle\\
&=\E \E_Y \nabla\cdot (f(X)\nabla g(\cos\theta\, X+\sin\theta\, Y))\\
&=\E \E_Y \langle \nabla f(X), \nabla g(X^\theta)\rangle+\cos\theta\, \E \E_Y f(X)\nabla \cdot \nabla g(X^\theta)\\
&=\E  \langle \nabla f(X), \nabla g(X^\theta)\rangle+\cos\theta\, \E  f(X)\nabla \cdot \nabla g(X^\theta).
\end{align*}
Likewise, we have the following
\begin{align*}
\E \langle f(X)\nabla g(X^\theta), Y\rangle&=\E \E_X \langle f(X)\nabla g(\cos\theta\, X+\sin\theta\, Y), Y\rangle\\
&=\sin\theta \, \E \E_X  f(X)\nabla \cdot \nabla g(\cos\theta\, X+\sin\theta\, Y)\\
&=\sin\theta \, \E f(X)\nabla \cdot \nabla g(\cos\theta\, X+\sin\theta\, Y).
\end{align*}
Therefore, we have
\[
\E f(X) \langle\nabla g(X^\theta), -\sin\theta\, X+\cos\theta\, Y \rangle =-\sin\theta \E \langle \nabla f(X), \nabla g(X^\theta)\rangle,
\]
and this finishes our proof.
\end{proof}

To apply Lemma \ref{lem:Gaussint}, we view $\bfW_k$ as $X$, and $\bfW'_k$ as $Y$ in \eqref{tmp:Deltak}. Let
\[
\bfW^{\theta,k}:=[\bfW_{1:k-1}, \cos\theta \bfW_k+\sin \theta \bfW_k',\bfW'_{k+1:T}].
\]
Then
\[
\Delta_k=g(\Phi_i (\bfW_{0:T}))g(\Phi_j(\bfW^{0,k}))-g(\Phi_i (\bfW_{0:T}))g(\Phi_j (\bfW^{\frac{\pi}{2},k}))
\]

Note that by the chain rule,
\[
\nabla_{k} g(\Phi_j(\bfW_{0:T}))=\nabla g (\Phi_j(\bfW_{0:T})) \nabla_k \Phi_j (\bfW_{0:T}).
\]
So by Lemma \ref{lem:Gaussint}, we have
\begin{align}
\notag
\E \Delta_k&=\int^{\frac\pi2}_0\sin \theta \E \langle \nabla g(\Phi_i(\bfW_{0:T})) \nabla_k \Phi_i (\bfW_{0:T}),  \nabla g(\Phi_j(\bfW^{\theta,k})) \nabla_k \Phi_j (\bfW^{\theta,k}) \rangle d\theta\\
\notag
&=\int^{\frac\pi2}_0\sin \theta \E \text{tr}\left( \nabla g(\Phi_j(\bfW^{\theta,k}))^T \nabla g(\Phi_i(\bfW_{0:T})) \nabla_k \Phi_i (\bfW_{0:T}) \nabla_k \Phi_j (\bfW^{\theta,k})^T \right) d\theta\\
\label{eqn:Edeltak}
&\leq \|\nabla g\|^2_\infty \int^{\frac\pi2}_0\sin \theta \E \left|\text{tr}\left( \nabla_k \Phi_i (\bfW_{0:T}) \nabla_k \Phi_j (\bfW^{\theta,k})^T \right)\right| d\theta.
\end{align}
Note that  $\nabla_k \Phi_i (\bfW_{0:T}) \nabla_k \Phi_j (\bfW^{\theta,k})^T\in \reals^{q\times q}$ is the $(i,j)$-th sub-block of
\begin{equation}
\label{tmp:Rterminal}
R^k:=\nabla_k \Phi (\bfW_{0:T}) \nabla_k \Phi (\bfW^{\theta,k})^T\in \reals^{qN\times qN}.
\end{equation}
So essentially we want bound \eqref{tmp:Rterminal} in each of its block.  To do that, we need to consider the derivative of the flow $\Psi_t$.

\subsection{Derivative flow}
\begin{lem}
\label{lem:derivative}
Fix any realization of $\bfW_{0:T}$ and $\bfW_{0:T}'$. Let $Q_h(s)= \exp(hG_* s) \exp(hG_* s)^T$. The $(i,j)$-th sub-block of $R^k$ is bounded by
\[
\|R^k_{i,j}\|_{F}\leq \|h\Sigma^2\|_F [Q(T-k)]_{i,j}\quad k\geq 1; \quad \|R^0_{i,j}\|_{F}\leq \|\Sigma^2_0\|_F [Q(T)]_{i,j}.
\]
\end{lem}
\begin{proof}
%[Proof of Lemma \ref{lem:derivative}]
We  will discuss only the $k\geq 1$ case, the analysis $k=0$ case is identical. For any $s\leq T-k$, we can extend the definition of $\bfX(t)$ to noninteger $t$, by letting
\[
\bfX(k+s)=\Psi_{sh}(\bfX(k+\lfloor s\rfloor))\quad \text{if}\,\, s\neq \lfloor s\rfloor.
\]
Here $\lfloor s\rfloor$ denotes the largest integer that is less than $s$.
The derivative flow of $\Psi_t$ can also be extended. We define the following $\reals^{qN\times qN}$ matrix:
\[
P(s)= \nabla \Psi_{(s-\lfloor s\rfloor)h}(\bfX(k+\lfloor s\rfloor))\cdots\nabla \Psi_h(\bfX(k))\Sigma_h.
\]
Then following the back propagation formula \eqref{eqn:gradientphi},  $\nabla_k\Phi=P(T-k)$. $P(s)$ also follows the following differential equation with initial condition $P(0)=\Sigma_h$,
\[
\dot{P}(s)=h G(k+s,\bfX(k+s)) P(s).
\]
$G(s,\bfx)\in \reals^{qN\times qN}$ is the Jacobian created by the ODE flow \eqref{sys:ODE}. So its $(i,j)$-th sub-block is given by
\[
G_{i,j}(s,\bfx)=\nabla_{\bfx_j} \bff_i(s,\bfx_{i-1}, \bfx_{i}, \bfx_{i+1})+\frac1N\nabla_{\bfx_j} h_i(s,\bfx_j).
\]
Clearly, it is a block-tridiagonal matrix.

Because of \eqref{tmp:Rterminal}, we are also interested in $\nabla_k \Phi (\bfW^{\theta,k})$. We consider a numerical sequence generated by \eqref{eqn:update}, but the driving noise is  $\bfW^{\theta,k}$. Fix a $\theta\in [0,\frac\pi2]$.
Let $\bfY(n)=\bfX(n)$ for $n=0,\ldots, k$,
\[
 \bfY(k+1)=\Psi_h(\bfX(k))+\Sigma_h \cos\theta \bfW_k+\Sigma_h\sin\theta \bfW'_k,
\]
and $\bfY(n+1)=\Psi_h(\bfY(n))+\Sigma_h \bfW'_n$ for $n\geq k+1$. We also define $\bfY(k+s)=\Psi_{sh}(\bfY(k+\lfloor s\rfloor))$ for noninteger time $s\neq \lfloor s\rfloor$. Define also the following matrix
\[
P'(s)= \nabla \Psi_{(s-\lfloor s\rfloor)h}(\bfY_{k+\lfloor s\rfloor})\cdots\nabla \Psi_h(\bfY_{k})\Sigma_h.
\]
It follows the ODE
\[
\dot{P}'(s)=h G(k+s,\bfY(k+s)) \dot{P}'(s).
\]
By the back propogation formula \eqref{eqn:gradientphi}, $P'(T-k)=\nabla_k \Phi (\bfW^{\theta,k})$.

Let  $R(s)=P(s)(P'(s))^T$. The quantity we are interested is  $R^k=R(T-k)$. $R(s)$ follows the following ODE:
\[
\dot{R}(s)=h G(\bfX_{k+s})  R(s)+hR(s) G(\bfY_{k+s})^T,\quad R(0)=\Sigma_h \Sigma^T_h=\Sigma_h^2.
\]
Here and below, we write $G(k+s,\bfX_{k+s})$ as $G(\bfX_{k+s})$ for simplicity.
Consider the $(i,j)$-th sub block of $R_{i,j}(s)$. In below, we use $\{AB\}_{i,j}$ to denote the $(i,j)$-th sub block of matrix $AB$.
We also define the inner product of matrices as $\langle A, B\rangle=\text{tr}(A^T B)$. Then  $\| R_{i,j}(s)\|^2_F=\langle R_{i,j}(s), R_{i,j}(s)\rangle $. We take time derivative,  and recall that $G$ is block tridiagonal, so
\begin{align*}
\frac{d}{ds}&\langle R_{i,j}(s), R_{i,j}(s)\rangle
=2h\langle \{G(\bfX_{k+s})  R(s)\}_{i,j}, R_{i,j}(s)\rangle
+2h\langle \{R(s) G(\bfY_{k+s})^T\}_{i,j}, R_{i,j}(s)\rangle\\
&=2h\langle G_{i,i}(\bfX_{k+s}) R_{i,j}(s), R_{i,j}(s)\rangle
+2h\langle R_{i,j}(s) G_{j,j}(\bfY_{k+s})^T, R_{i,j}(s)\rangle\\
&\quad+2h\sum_{m=\pm1} \langle G_{i,i+m}(\bfX_{k+s}) R_{i+m,j}(s), R_{i,j}(s)\rangle+2h\sum_{m=\pm1}\langle R_{i,j+m}(s) G_{j,j+m}(\bfY_{k+s})^T , R_{i,j}(s)\rangle.
\end{align*}
Because of the covariance propagation assumption in Theorem \ref{thm:main}, and an elementary \ref{lem:matrix} for matrix inner product, we can bound each term as below
\[
\langle G_{i,i}(\bfX_{k+s}) R_{i,j}(s), R_{i,j}(s)\rangle\leq [ G_*(s+k)]_{i,i}\|R_{i,j}(s)\|_F^2,
\]
\[
\langle R_{i,j}(s) G_{j,j}(\bfY_{k+s})^T, R_{i,j}(s)\rangle\leq [ G_*(s+k)]_{j,j}\|R_{i,j}(s)\|_F^2,
\]
\[
\langle G_{i,i+m}(\bfX_{k+s}) R_{i+m,j}(s), R_{i,j}(s)\rangle \leq [ G_*(s+k)]_{i,i+m}\|R_{i,j}(s)\|_F\|R_{i+m,j}(s)\|_F,
\]
\[
\langle R_{i,j+m}(s) G_{j,j+m}(\bfY_{k+s})^T , R_{i,j}(s)\rangle \leq [ G_* (s+k)]_{j,j+m}\|R_{i,j}(s)\|_F\|R_{i,j+m}(s)\|_F.
\]
In conclusion, we have the following with the dependence of $G_*$ on $s+k$ suppressed,
\begin{align*}
\frac{d}{ds}\| R_{i,j}(s)\|^2_F
\leq &2([ G_*]_{i,i}+[ G_*]_{j,j})h\|R_{i,j}(s)\|_F^2
+2\sum_{m=\pm 1}[ G_*]_{i,i+m}h\|R_{i,j}(s)\|_F\|R_{i+m,j}(s)\|_F\\
&+2\sum_{m=\pm 1} [ G_*]_{j+m,j}h\|R_{i,j}(s)\|_F\|R_{i,j+m}(s)\|_F.
\end{align*}
Since $\frac{d}{ds}\| R_{i,j}(s)\|^2_F=2\| R_{i,j}(s)\|_F\frac{d}{ds}\| R_{i,j}(s)\|_F$, we have
\begin{align*}
\frac{d}{ds}\| R_{i,j}(s)\|_F
\leq &h([ G_*]_{i,i}+[ G_*]_{j,j})\|R_{i,j}(s)\|_F+h\sum_{m=\pm 1}[ G_*]_{i,i+m}\|R_{i+m,j}(s)\|_F+[ G_*]_{j+m,j}\|R_{i,j+m}(s)\|_F.
\end{align*}
Next we consider $Q_h(s)$. Note that $\|h\Sigma^2\|_F[Q_h(0)]_{i,j}=\|R_{i,j}(0)\|_F$. Moreover, $Q_h$ follows the ODE
\[
\frac{d}{ds}Q_h(s)=hG_*Q_h(s)+hQ_h(s) G_*^T.
\]
 Use the fact that $G_*$ is tridiagonal,
\begin{align*}
\frac d {ds} [Q_h(s)]_{i,j}&=h [G_*Q_h(s)]_{i,j}+h[Q_h(s) G^T_*]_{i,j}\\
&=h([G_*]_{i,i}+[G_*]_{j,j}) [Q_h(s)]_{i,j}+h\sum_{m=\pm 1}[G_*]_{i,i+m} [Q_h(s)]_{i+m,j} +[G_*]_{j,j+m} [Q_h(s)]_{i,j+m}.
\end{align*}
If we view $[Q_h(s)]_{i,j}$ as a linear ODE, the non diagonal term of this ODE system , $[G_*]_{i,i+m}$ and $[G_*]_{j,j+m}$, are all nonnegative.
So by Lemma \ref{lem:exp}, we have that
\[
\| R_{i,j}(s)\|_F\leq [Q(s)]_{i,j}\|h\Sigma^2\|_F.
\]
This leads to our claims.
\end{proof}

\begin{lem}
\label{lem:matrix}
Suppose $A+A^T\preceq 2\lambda_A I$, then $\langle AX, X\rangle\leq \lambda_A \langle X, X\rangle$, and $\langle XA, X \rangle\leq \lambda_A\langle X, X\rangle$.
\end{lem}
\begin{proof}
Let $v_i$ be the right eigenvectors of $X$, so $Xv_i=\lambda_i v_i$, then
\[
\langle AX, X\rangle =\text{tr}(AXX^T)=\text{tr}(X^T A X)=\sum_i v_i^* X^T AX v_i
=\sum_{i}|\lambda_i|^2 v_i^* Av_i\leq \lambda_A \sum_{i} |\lambda_i|^2=\lambda_A \langle X, X\rangle.
\]
For the second claim, let $v_i$ be the right eigenvectors of $X^T$, so $X^Tv_i=\lambda_i v_i$
\[
\langle XA, X\rangle =\text{tr}(XAX^T)=\sum_i v_i^* X^T AX v_i
=\sum_{i}|\lambda_i|^2 v_i^* Av_i\leq \lambda_A \sum_{i} |\lambda_i|^2=\lambda_A \langle X, X\rangle.
\]
\end{proof}

\begin{lem}
\label{lem:exp}
If $A\in \reals^{d\times d}$ has all non-diagonal entries being nonnegative, then $\exp(sA)$ has all entries being nonnegative for any $s$. In particular, if $X(s)\in \reals^d$ and $Y(s)\in\reals^d$ satisfies
\[
\frac d{ds}X(s)\leq A X(s),\quad \frac d{ds}Y(s)=A Y(s),
\]
where $\leq$ is interpreted entry-wise. Then $X(0)\leq Y(0)$  leads to  $X(s)\leq Y(s)$ entry-wise.
\end{lem}
\begin{proof}
The statement is elementary if all of $A$ diagonal entries are also nonnegative, since
\[
\exp(sA)=I+sA+ \frac1{2!} s^2 a^2+\frac1{3!} s^3 A^3+\cdots.
\]
On the other hand, we can always find a $\lambda$ so that $\lambda \bfI+sA$ has all  diagonal entries being nonnegative. Because $\bfI$ commute with $sA$, so
\[
\exp(sA)=\exp(-\lambda \bfI)\exp(\lambda \bfI+sA)=e^{-\lambda} \exp(\lambda \bfI+sA),
\]
and all entries of $\exp(sA)$ are nonnegative.

For the ODE system, we first consider the case when all entries of $A$ are positive.  Define $\Delta(s)=Y(s)-X(s)$. $\Delta(0)\geq \mathbf{0}$ entry-wise. Also when $\Delta(s)\geq 0$,
\[
\frac d{ds}\Delta(s)\geq A \Delta(s)\geq \mathbf{0}.
\]
meaning $\Delta(s)$ is increasing entry-wise. Therefore $\Delta (s)\geq \mathbf{0}$ entry-wise for $s\geq 0$.

If the diagonal terms of $A$ are negative, we can consider $X_\lambda(t)=e^{\lambda t}X(t)$ and $Y_\lambda(t)=e^{\lambda t} Y(t)$, then
\[
\frac d{ds}X_\lambda (s)\leq (A+\lambda\bfI) X(s),\quad \frac d{ds}Y(s)=(A+\lambda\bfI) Y(s).
\]
By our previous analysis, $X_\lambda(t)\leq Y_\lambda(t)$ entry-wise, which leads to $X(t)\leq Y(t)$ entry-wise.
\end{proof}

\subsection{Summarizing argument}
\begin{proof}[Proof of Theorem \ref{thm:main}]
Recall in  \eqref{eqn:Edeltak}, we have
\[
|\E \Delta_k |\leq \|\nabla f\|^2_\infty \int^{\frac{\pi}{2}}_0\sin\theta |\E \text{tr} (R^k_{i,j})|d\theta.
\]
Note that $R^k_{i,j}$ is $q\times q$ dimensional, so by Cauchy-Schwartz,  $|\text{tr} (R^k_{i,j})|\leq \sqrt{q}\|R^k_{i,j}\|_F.$
Lemma \ref{lem:derivative} shows that
\[
\|R^k_{i,j}\|_{F}\leq h\|\Sigma^2\|_F [Q_h(T-k)]_{i,j}\quad k\geq 1; \quad \|R^0_{i,j}\|_{F}\leq \|\Sigma_0^2\|_F [Q_h(T)]_{i,j}.
\]
Therefore for $k\geq1$,
\[
|\E \Delta_k |\leq  h\sqrt{q}\|\nabla g\|^2_\infty  \|\Sigma^2\|_F [Q_h(T-k)]_{i,j} \int^{\frac{\pi}{2}}_0\sin\theta d\theta=
\sqrt{q}h\|\nabla g\|^2_\infty  \|\Sigma^2\|_F [Q_h(T-k)]_{i,j},
\]
and likewise $|\E \Delta_0 |\leq \sqrt{q}\|\nabla g\|^2_\infty  \|\Sigma^2_0\|_F [Q_h(T)]_{i,j}$. Consequentially, we have
\begin{equation}
\label{eqn:final}
|\text{cov} (g(\bfX_i (T)), g(\bfX_j(T)))|\leq
 \sqrt{q}\|\nabla g\|^2_\infty \left(\|\Sigma_0^2\|_FQ_h(T)+h\|\Sigma^2\|_F\sum_{k=1}^T \E [Q_h(T-k)]_{i,j}\right).
\end{equation}
When we let $h\to 0$, $\bfX (T)\to \bfx(t)$ by Lemma \ref{lem:htozero}. Then
\begin{align*}
&\left|\E g(\bfX_i (T)) g(\bfX_j(T))- \E g(\bfx_i (t))g(\bfx_j(t))\right|\\
&\leq
\E |g(\bfx_ i(t)) (g(\bfX_j(T))-g(\bfx_j(t))) |+ \E |g(\bfx_ j(t)) (g(\bfX_i(T))-g(\bfx_i(t))) |\\
&\quad+ \E |(g(\bfX_i (T))-g(\bfx_i (T))) (g(\bfX_j(T))-g(\bfx_j(t))) |\\
&\leq \sqrt{\E |g(\bfx_i (T))|^2 \E |g(\bfX_j(T))-g(\bfx_j(t))|^2}+ \sqrt{\E |g(\bfx_ j(t))|^2 \E|g(\bfX_i(T))-g(\bfx_i(t))|^2}\\
&\quad+ \sqrt{\E |g(\bfX_i(T))-g(\bfx_i(t))|^2 \E |g(\bfX_j(T))-g(\bfx_j(t))|^2}
\end{align*}
By the Lipschitzness of SDE coefficient, it is standard to show that $\E |g(\bfx_i(t))|^2\leq  \|\nabla g\|_\infty^2 \E \|\bfx_i(t))\|^2$ is bounded. By Lemma \ref{lem:htozero} we know
\[
\E |g(\bfX_j(T))-g(\bfx_j(t))|^2\leq \E \|\nabla g\|_\infty^2 \|\bfX_j(T)-\bfx_j(t)\|^2=O(h).
\]
Also note that $Q_h(n)=Q(hn)$, by letting $h\to 0$, we have our claim.
\end{proof}
\subsection{Continuous time random walk}

\begin{lem}
The $Q(s)$ in Theorem \ref{thm:heat} has the following upper bound
\[
[Q(s)]_{i,j}\leq 2\exp\left(\lambda_\beta s\right)\left(\exp(-\beta \bfd(i,j))+\frac{(1+e^{-\beta})(e^{\lambda_H s}-1)}{(1-e^{-\beta})N }\right),
\]
and
\begin{align*}
\int^t_0 [Q(s)]_{i,j}ds&\leq     \frac{2\exp\left(\lambda_\beta t\right)-2}{\lambda_\beta}\left( e^{-\beta \bfd(i,j)}-\frac2{(1-e^{-\beta}) N}\right)+\frac{2(\exp\left(\eta_\beta  t \right)-1)(1+e^{-\beta})}{\eta_\beta (1-e^{-\beta}) N}.
\end{align*}
\end{lem}
\begin{proof}
Let $F_H=F_*-\lambda_G \bfI_{N}$, then its entries are
\[
[F_H]_{i,i}=-2\lambda_F,\quad [F_H]_{i,i\pm 1}=\lambda_F.
\]
Then $\exp(F_H s)=e^{s\lambda_G }\exp(F_* s)$. Moreover, because $H_* F_H=F_H H_*=\mathbf{0}$, so  $\exp(G_H s) =\exp(F_H s)\exp(H_* s)$. We first study the behavior of $\exp(F_Hs)$.

Let $Z_t$ be a continuous time random walk on  $\mathbb{Z}$,  with rate $\lambda_{z,z\pm1}=\lambda_H.$ Denote its projection on the modular space $\mathbb{Z}/N\mathbb{Z}$ as $Y_t$, which becomes  a continuous time random walk on $\mathbb{Z}/N\mathbb{Z}$. Clearly $F_H$ is the transition rate matrix of $Y_t$. Therefore,
\[
[\exp(F_Hs)]_{i,j}=\Prob(Y_s=j|Y_0=i).
\]
Then by the Markov property and the translation invariance of random walk,
\[
[\exp(F_Hs)]_{i,j}=\sum_{n=-\infty}^\infty\Prob(Z_s=j-i +nN|Z_0=0)\leq \Prob(|Z_s|\geq \bfd(i,j)|Z_0=0)
=2\Prob(Z_s> \bfd(i,j)|Z_0=0).
\]
Next we investigate the Laplace transform $e^{\beta z}$. Applying the generator of $Z_t$
\[
\mathcal{A} e^{\beta z}=\lambda_F e^{\beta (z+1)}+\lambda_F e^{\beta (z-1)}-2\lambda_F e^{\beta z}
=\lambda_F (e^\beta+e^{-\beta}-2) e^{\beta z}.
\]
By Dynkin's formula,
\[
\frac{d}{dt}\E e^{\beta Z_t}=\lambda_F (e^\beta+e^{-\beta}-2)  \E e^{\beta Z_t},
\]
therefore, $\E (e^{\beta Z_t}|Z_0=0)=\exp(\lambda_F t (e^\beta+e^{-\beta}-2))$. By Markov inequality,
\[
\Prob(Z_s> \bfd(i,j)|Z_0=0)\leq 2\exp\left(\lambda_F s (e^\beta+e^{-\beta}-2)-\beta  \bfd(i,j)\right).
\]
This leads to our first claim on $[\exp (F_H s)]_{i,j}\leq 2\exp\left(s\lambda_G+ \lambda_F s (e^\beta+e^{-\beta}-2)-\beta  \bfd(i,j)\right)$.

 Next we try to compute $\exp(H_* s)$. Because $(H_*/\lambda_H)^n=H_*/\lambda_H=\frac1N\mathbf{1} \cdot \mathbf{1}^T$ for any $n\geq 0$, so
\[
\exp(H_* s)=\sum_{n=0}^\infty \frac{s^n\lambda_H^n}{n!}\left(\frac{H_*}{\lambda_H}\right)^n = \frac{e^{\lambda_H s}-1}{\lambda_H} H_*+\bfI_N.
\]
Then
\begin{align*}
[Q(s)]_{i,j} &\leq \sum_{k} [\exp (F_H s)]_{i,k}[\exp(H_* s)]_{k,j}\\
&=[\exp (F_H s)]_{i,j}+\frac1N(e^{\lambda_H s}-1)\sum_{k=1}^N [\exp (F_H s)]_{i,k}\\
&\leq 2\exp\left(\lambda_Gs+ \lambda_F s (e^\beta+e^{-\beta}-2)\right)\left(\exp(-\beta \bfd(i,j))+\frac{e^{\lambda_H s}-1}{N }\sum_{k=1}^N \exp(-\beta \bfd(i,k))\right)\\
&\leq 2\exp\left(\lambda_\beta s\right)\left(\exp(-\beta \bfd(i,j))+\frac{(1+e^{-\beta})(e^{\lambda_H s}-1)}{(1-e^{-\beta})N }\right).
\end{align*}

 Its integral can be further bounded by
\begin{align*}
\int^t_0 [Q(s)]_{i,j}ds &\leq 2\int^t_0\exp\left(\lambda_\beta s\right)\left(\exp(-\beta \bfd(i,j))+\frac{(1+e^{-\beta})(e^{\lambda_H s}-1)}{(1-e^{-\beta})N }\right)ds\\
&= \frac{2\exp\left(\lambda_\beta t\right)-2}{\lambda_\beta}\left( e^{-\beta \bfd(i,j)}-\frac{1+e^{-\beta}}{(1-e^{-\beta}) N}\right)+\frac{2(\exp\left(\eta_\beta  t \right)-1)(1+e^{-\beta})}{\eta_\beta (1-e^{-\beta}) N}.
\end{align*}
\end{proof}

\section{Analytic solution of the linear model}
\label{app:linear}
For the linear model \eqref{linear_model}, collecting all the state variables $\mathbf{u} = (u_1, u_2, \ldots, u_N)$, the abstract form the the system is given by
\begin{equation}\label{linear_model_abstract}
  \frac{d\mathbf{u}}{dt} = \mathbf{A}\mathbf{u} + \boldsymbol\Sigma \dot{\mathbf{W}}_u.
\end{equation}
According to \eqref{linear_model}, the coefficient matrices $\mathbf{A}$ and $\boldsymbol\Sigma$ are given by
\begin{equation*}
\begin{split}
    \mathbf{A} =& \left(
                   \begin{array}{cccccc}
                     -a &   &   &   & &\\
                       & -a &   &   & &\\
                       &  & -a &   &   & \\
                       &   & &\ddots  && \\
                       &   &   &  &-a& \\
                       &   &   &  & &-a \\
                   \end{array}
                 \right) + d_u \left(
                   \begin{array}{cccccc}
                     -2 & 1  &          &           &       &1\\
                      1 & -2 & 1        &           &       &\\
                        &  1 & -2       & 1         &       &\\
                        &    &  \ddots  & \ddots    &\ddots &\\
                        &    &          & 1         & -2    & 1\\
                      1 &    &          &           &  1    & -2\\
                   \end{array}
                 \right)\\
    &+ \frac{w}{N}\left(
                   \begin{array}{cccccc}
                     1 & 1 & \ldots  & \ldots & 1 & 1 \\
                     1 & 1 & \ldots  & \ldots & 1 & 1 \\
                     \vdots & \vdots & \ddots & \ddots & \vdots & \vdots \\
                     \vdots & \vdots & \ddots & \ddots & \vdots & \vdots \\
                     1 & 1 & \ldots  & \ldots & 1 & 1 \\
                     1 & 1 & \ldots  & \ldots & 1 & 1 \\
                   \end{array}
                 \right)- w \left(
                   \begin{array}{cccccc}
                     1 &   &   &   &   &   \\
                       & 1 &   &   &   &   \\
                       &   & 1 &   &   &   \\
                       &   &   & 1 &   &   \\
                       &   &   &   & 1 &   \\
                       &   &   &   &   & 1 \\
                   \end{array}
                 \right)   \\
  \boldsymbol\Sigma =& \sigma_u\left(
                                \begin{array}{cccc}
                                  1 &  & & \\
                                    & 1 & &\\
                                    & &\ddots &  \\
                                    &  && 1 \\
                                \end{array}
                              \right)
\end{split}
\end{equation*}
It is easy to verify that $\mathbf{A}$ is a negative definite matrix. Therefore, eigenvalue decomposition allows the following
\begin{equation*}
  \mathbf{A} = \mathbf{Q}\boldsymbol\Lambda\mathbf{Q}^T,
\end{equation*}
where $\boldsymbol\Lambda$ is a diagonal matrix with all diagonal entries $\lambda_1,\ldots,\lambda_N$ being negative.

The solution of $\mathbf{u}$ is given by
\begin{equation}\label{soln_u}
  \mathbf{u}(t) = \mathbf{u}(0)e^{\mathbf{A}t} + e^{\mathbf{A}t}\int_0^t e^{-\mathbf{A}t}\boldsymbol\Sigma d\mathbf{W}_u.
\end{equation}
The time evolution of the covariance matrix is given by
\begin{equation}\label{cov_u}
\begin{split}
  \langle\mathbf{u}_t^{\prime2} \rangle &= \langle\mathbf{u}_0^{\prime2} \rangle e^{2\mathbf{A}t} + \sigma_u^2\int_0^t e^{2\mathbf{A}(t-s)} ds\\
  &=\langle\mathbf{u}_0^{\prime2} \rangle \mathbf{Q}e^{2\boldsymbol\Lambda t}\mathbf{Q}^T + \sigma_u^2\int_0^t \mathbf{Q}e^{2\boldsymbol\Lambda (t-s)}\mathbf{Q}^T ds\\
  &=\langle\mathbf{u}_0^{\prime2} \rangle \mathbf{Q}e^{2\boldsymbol\Lambda t}\mathbf{Q}^T + \sigma_u^2 \mathbf{Q}\left\{\frac{e^{2\lambda_it}-1}{2\lambda_i}\right\}\mathbf{Q}^T, \\
\end{split}
\end{equation}
where
\begin{equation*}
  \left\{\frac{e^{2\lambda_it}-1}{2\lambda_i}\right\} = \left(
                                                          \begin{array}{ccc}
                                                            (e^{2\lambda_1t}-1)/(2\lambda_1) &   &  \\
                                                              & \ddots &  \\
                                                              &   & (e^{2\lambda_Nt}-1)/(2\lambda_N) \\
                                                          \end{array}
                                                        \right).
\end{equation*}

\bibliographystyle{plain}
\bibliography{references}
\end{document}